\documentclass[twoside, 11pt]{article}
\usepackage[preprint]{jmlr2e}
\usepackage{amsfonts,amssymb,amsmath,bbm, xcolor}
\usepackage{amssymb}
\usepackage[shortlabels]{enumitem}
\usepackage[normalem]{ulem}
\usepackage{cancel}
\usepackage{scalerel}
\usepackage{multicol}
\usepackage{xfrac}

\usepackage{mathrsfs}%,showlabels}
\usepackage[toc,page]{appendix}

\usepackage{bm}
\usepackage{graphicx}
\usepackage{algorithm,algorithmic}
\usepackage{subcaption}
\usepackage{placeins}

\usepackage{mismath}

\usepackage{thmtools}
\usepackage{thm-restate}

\DeclareMathOperator*{\argmin}{arg\,min}
\DeclareMathOperator{\Tr}{Tr}
\newtheorem{assumption}{Assumption}

\advance\hoffset by 5mm

\newcommand{\Rm}{{\mathbb R}}
\newcommand{\eps}{\varepsilon}
\newcommand{\commentout}[1]{}
\newcommand{\cL}{\mathcal{P}}

\renewcommand{\phi}{\varphi}

\newtheorem{thm}{Theorem}[section]

\newcommand{\be}{\begin{equation}}
\newcommand{\ee}{\end{equation}}

\newcommand{\bal}{\begin{aligned}}
\newcommand{\enbal}{\end{aligned}}

\newcommand{\one}{{\mathbbm{1}}}

\newcommand{\bhat}[1]{\expandafter\hat#1} 

\renewcommand{\d}{\partial}

\numberwithin{equation}{section}

\usepackage{lastpage}
\ShortHeadings{Convergence of the two-timescale Q-learning}{An, Lu, Wu, and Xiang}
\firstpageno{1}

\begin{document}
\title{Why does the two-timescale Q-learning  converge to different mean field solutions? A unified convergence analysis}

%\author{}
%\affil{Duke University}

% \author[1]{Jing An}
% %\ead{jing.an@duke.edu}
% \affil[1]{Departments of Mathematics, Duke University}
% \author[2]{Jianfeng Lu}
% %\ead{jianfeng@math.duke.edu}
% \affil[2]{Departments of Mathematics, Physics, and Chemistry, Duke University}

% \author[3]{Yue Wu}
% %\ead{ywudb@connect.ust.hk}

% \affil[3]{Department of Mathematics, The Hong Kong University of Science and Technology}
            
% \author[3,4]{Yang Xiang}
% %\ead{maxiang@ust.hk}

% \affil[4]{
% HKUST Shenzhen-Hong Kong Collaborative Innovation Research Institute}

\author{\name Jing An \email jing.an@duke.edu \\
       \addr Department of Mathematics,\\
       Duke University 
       \AND
\name Jianfeng Lu \email
jianfeng@math.duke.edu \\
 \addr Department of Mathematics, Physics, and Chemistry,\\
       Duke University 
       \AND
\name Yue Wu \email
ywudb@connect.ust.hk\\
\addr Department of Mathematics, \\
The Hong Kong University of Science and Technology
\AND
\name Yang Xiang \email
maxiang@ust.hk \\
\addr
Department of Mathematics, \\
The Hong Kong University of Science and Technology \\
and \\
HKUST Shenzhen-Hong Kong Collaborative Innovation Research Institute
       }

\editor{My editor}

\maketitle

\begin{abstract}
% We revisit the unified two-timescale Q-learning algorithm introduced in \cite{angiuli2022unified}. 
% This two-timescale Q-learning algorithm alone can solve  Mean Field Game (MFG) and  Mean Field Control problems, simply by tuning the ratio of two learning rates for mean field distribution and the Q function respectively. In this paper, considering fixed learning rates, we provide a thorough theoretical explanations for bifurcated numerical results of the algorithm, by establishing a diagram linking the continuous-time mean field problems  to corresponding discrete-time Q functions where the algorithm is based on.  In particular, we construct a Lyapunov function involving both mean field distribution and the Q function iterates to provide a unified single-timescale convergence of the algorithm covering all range of learning rates. 
We revisit the unified two-timescale Q-learning algorithm as initially introduced by  \cite{angiuli2022unified}. This algorithm demonstrates efficacy in solving mean field game (MFG) and mean field control (MFC) problems, simply by tuning the ratio of two learning rates for mean field distribution and the Q-functions respectively. In this paper, we provide a comprehensive theoretical explanation of the algorithm's bifurcated numerical outcomes under fixed learning rates. We achieve this by establishing a diagram that correlates continuous-time mean field problems to their discrete-time Q-function counterparts, forming the basis of the algorithm. Our key contribution lies in the construction of a Lyapunov function integrating both mean field distribution and Q-function iterates. This Lyapunov function facilitates a unified convergence of the algorithm across the entire spectrum of learning rates, thus providing a cohesive framework for analysis.
\end{abstract}
 

\section{Introduction}
Reinforcement learning (RL) is a dynamic machine learning technique formalized through the framework of Markov Decision Processes (MDP), wherein an agent learns through interaction within an environment, relying on trial and error and feedback derived from its own actions and experiences \citep{sutton2018reinforcement}. RL has been prominent in artificial intelligence research in past decades and yields breakthroughs across diverse domains ranging from robotics \citep{kober2013reinforcement}, classical games  \citep{mnih2013playing, silver2016mastering}, to autonomous driving \citep{kiran2021deep}. RL is closely related to the optimal control problems in the sense that it optimizes the decision-making processes by maximizing long-term cumulative rewards or minimizing cumulative costs under accessible policies \citep{bertsekas2019reinforcement}. Multi-agent reinforcement learning (MARL) extends the classical RL to scenarios involving multiple agents interacting within a shared environment, and we refer to survey works \citep{busoniu2008comprehensive, zhang2021multi} for its fundamental background. Despite its empirical success, the scalability of MARL with respect to the number of agents remains to be a key issue \citep{hernandez2019survey}.

One approach to tackle the curse of scalability is to consider MARL in the regime with a large number of homogeneous agents. In this paradigm, mean field formulations provide a mathematical framework to model  and analyze large-scale interacting particle systems independent of the number of agents $N$. Particularly, we focus on mean field game (MFG) and mean field control (MFC) problems as their theory has been developed rapidly in recent years. Mean field games, initially introduced by \cite{Larsy2007} and \cite{caines2006large}, are non-cooperative $N$-player games aim to find a Nash equilibrium where no individual agent can unilaterally improve the outcome by changing strategies. On the other hand, a mean field control problem has a central planner to find the collective optimum in a cooperative game within a large population. We refer to books \cite{Bensoussan2013} and \cite{carmona2018probabilistic} for  further details of both MFG and MFC.

In the past years, solving stochastic control and games using model-free RL algorithms has gained a lot of interests, if one wants the agent to learn the optimal policy by directly interacting with the system without inferring the model parameters. We refer to \citep{hu2023recent, lauriere2022learning} for a comprehensive review of recent developments.
To learn MFG and MFC solutions, there are numerous algorithms available including but not limited to policy gradient based methods \citep{bhandari2024global, carmona2019linear, williams1992simple}, actor-critic methods for linear-quadratic models \citep{fu2019actor, yang2018mean,  wang2021global}, fixed point iterations relying on entropy-regularization \citep{cui2021approximately, guo2022entropy}, and value-based RL methods such as Q-learning \citep{angiuli2022unified, angiulia2023reinforcement, carmona2023model, guo2019learning, mguni2018decentralised, SubramanianMahajan-2018-RLstatioMFG,zaman2023oracle, gu2021mean}.

In this paper, we focus on the work by \cite{angiuli2022unified} that proposed a unified RL algorithm combining the classical Q-learning updates \citep{watkins1989learning, watkins1992q} with the two-timescale approach \citep{borkar1997stochastic}. This two-timescale Q-learning algorithm updates 
the mean field distribution and the value function iteratively, and can converge to either the MFG or MFC solutions by adjusting the ratio of associated Robbins-Monro learning rates to zero or infinity. A natural question to ask is why this simple two-timescale algorithm can produce bifurcated numerical results by just tuning two learning rates. We attempt to answer this question by:
\begin{enumerate}
    \item Building a complete roadmap connecting the discrete-time Q-learning algorithm to continuous-time Hamilton-Jacobi-Bellman (HJB) equations for both MFG and MFC. The corresponding HJB equations for MFG and MFC are different depending on whether the population distribution is fixed, while such dependence is not explicitly captured in the two-timescale Q-learning algorithm. 
    \item Providing a unified convergence of the two-timescale Q-learning algorithm covering all choices of fixed learning rates. Rather than focusing on the extreme regimes of learning rates ratios and qualitative analysis, we aim to explain the algorithm's bifurcated numerical behaviors quantitatively using the unified convergence result.
\end{enumerate}

\subsection{Our contributions}
For the first part of our work, we start from the continuous-time MFG and MFC value functions under stochastic control with infinite time horizon, and we give formulations of corresponding discrete-time value and Q-functions which the two-timescale Q-learning algorithm is built on. We provide a sequence of approximation results to verify connections in Fig.~\ref{fig:roadmap} with respect to the time discretization $h$.

For the second part, we construct a Lyapunov function integrating both mean field distribution and Q-function iterates from the two-timescale Q-learning algorithm and prove its convergence quantitatively. Our approach takes generic assumptions on the cost function and the transition kernel, in addition to assuming the transition kernel satisfying a uniform Doeblin's condition. The contraction of the constructed Lyapunov function exhibits explicit dependence on the two-timescale learning rates, thus explains how the two-timescale Q-learning algorithm can produce different solutions by tuning learning rates.

\begin{figure}[htb]
		\centering
		\includegraphics[width=12cm]{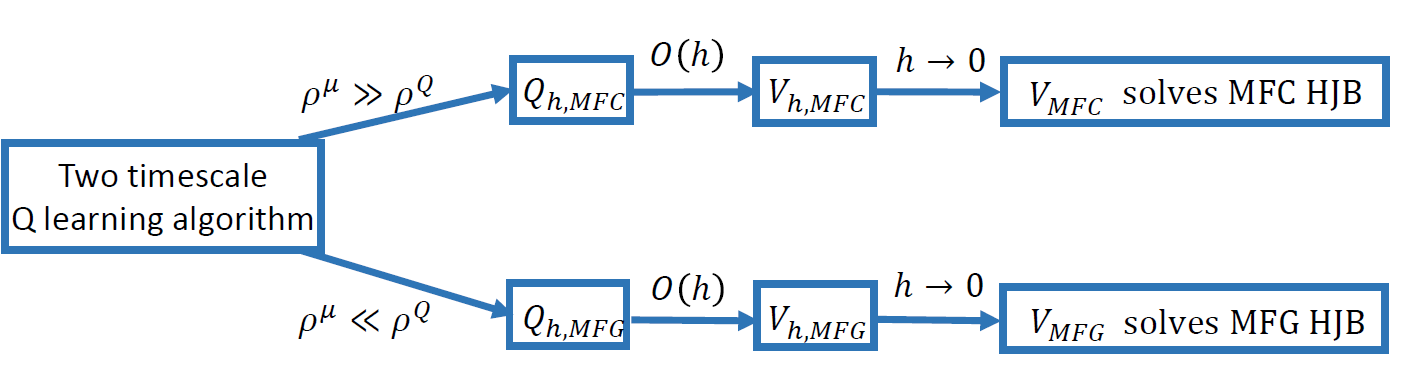}
         \caption{ 
The diagram that links two-timescale Q-learning algorithm with optimal value functions solving HJB equations.
%after taking infimum with respect to control.
}
           \label{fig:roadmap}
\end{figure}

\subsection{Related works}
We mention several works related to our approach. We build the convergence diagram as in Fig. \ref{fig:roadmap} since there is no continuous-time limit for the Q-learning iterations \citep{tallec19a}.  In the
continuous-time setting where we can seek for differences of MFG and MFC in the HJB equations, the Q-function from the algorithm becomes ill-posed and collapses to the value function that is independent of actions. To analyze the continuous-time counterpart of Q-learning, \cite{Kim2021hjb} restricts the action process to be Lipschitz continuous so that Q-learning becomes a policy evaluation problem with
the state-action pair as the new state variable. \cite{jia2023q} and \cite{wang2020reinforcement} consider and analyze the entropy-regularized, exploratory diffusion process formulation which approximates the classical Q-function independent of time discretization. We list a few other papers \citep{kim2020hamilton, gu2016continuous, jiang2015global, palanisamy2014continuous, vamvoudakis2017q} regarding general continuous-time RL. 

On the other hand, our unified convergence of the two-timescale Q-learning algorithm takes motivations from the community that studies bi-level optimization. One problem somewhat related to setups in this paper is the linear quadratic
regulator (LQR) in RL, and there have been many works studying two-timescale actor-critic algorithms for solving the LQR problem \citep{konda1999actor,ZhouLu2023, yang2019provably, zeng2021two}. In particular, our Lyapunov function construction is inspired by \cite{ZhouLu2023} who constructed a Lyapunov
function involving both the critic error and the actor loss, although our goal is different from \cite{ZhouLu2023}: We try to find the fixed point of mean field problems, while \cite{ZhouLu2023} aims to solve optimization problems. 

We also mention that a recent paper by \cite{angiuli2023convergence} applies the theory of stochastic approximation \citep{borkar1997stochastic} to the two-timescale Q-learning algorithm, and they showed the algorithm convergence in extreme regimes where the ratio of learning rates is either zero or infinity. Compared to \cite{angiuli2023convergence}, our approach of using the Lyapunov function is new and takes  care of all ranges of learning rate ratios. Moreover, the convergence in \cite{angiuli2023convergence} is qualitative while our result is quantitative.

\subsection{Organization}
The paper is organized as follows. 
In Section \ref{sec:background}, we review the formulations of the mean field game and mean field control problem that are the focus of our study.
In Section \ref{sec:problem}, we outline the discrete value functions and Q-functions for MFG and MFC in bounded state and action spaces, and we review the continuous-time value functions solving the HJB equations derived in unbounded state and action spaces. A sequence of approximation errors between various formulations are provided. In Section \ref{sec:alg}, we revisit the two-timescale Q-learning algorithm and illustrate its bifurcated numerical behaviors by a toy example.
In Section \ref{sec:conv}, we conduct the unified convergence analysis for the two-timescale Q-learning algorithm. Lastly,
the numerical experiments that verifying the algorithm are provided in Section \ref{sec:numeric}.

\paragraph{Notations}
Throughout the paper, we use $\|\cdot\|$ to denote the Euclidean norm;  \\$\|Q(\cdot, \cdot )\|_\infty:=\sup_{x\in\mathcal{X}, a\in\mathcal{A}} |Q(x,a)|$; $\|\mu\|_p =\big( \sum_{x\in\mathcal{X}} \mu(x)^p\big)^{1/p}$, and the total variation norm is $\|\mu\|_{\text{TV}} = \sup_{A\subseteq \mathcal{X}}\big|\sum_{x\in A}\mu(x)\big|$.

\paragraph{Acknowledgments}
JA would like to express thanks to Mo Zhou, Lei Li, Yingzhou Li, and Jiequn Han for fruitful discussions. This work was done during YW’s visit of Duke University and Rhodes Information Initiative at Duke. YX was partially supported by the Project of Hetao Shenzhen-HKUST Innovation Cooperation Zone HZQB-KCZYB-2020083.

\section{Background}\label{sec:background}
We consider the Markov Decision Process (MDP) \citep{bellman1957markovian, watkins1989learning} with finite state and action spaces, which we denote by $\mathcal{X}$ and $\mathcal{A}$ respectively. $\mathcal{P}(\mathcal{X})$ is the space of probability measures on $\mathcal{X}$.  %with finite second moments. 
The transition probability kernel can be viewed as a function 
\begin{equation}
    p: \mathcal{X}\times \mathcal{X}\times \mathcal{A}\times \mathcal{P}(\mathcal{X})\to [0,1],\quad (x,x',a,\mu)\mapsto p(x'\mid x,a,\mu),
\end{equation}
which is, under the population distribution $\mu$, the probability of jumping from state $x$ to state $x'$ using action $a$. Let $f:  \mathcal{X} \times  \mathcal{A} \times \mathcal{P}( \mathcal{X})\to \mathbb{R}_+$ be a running cost function. $f(x,a,\mu)$ can be interpreted as the one-step cost incurred by an agent at state $x$ to take an action $a$, when the population distribution is $\mu$.

There are various formulations of MFG and MFC problems available in the literature. In \cite{angiuli2022unified}, three formulations in the infinite horizon were presented:   asymptotic, non-asymptotic, and stationary. We will focus on the asymptotic formulations given in \cite{angiuli2022unified}, since then the problem faced by an infinitesimal agent among the crowd can be viewed as a MDP
parameterized by the population distribution. We refer other infinite time horizon formulations to \cite{angiuli2022unified} and the finite time horizon version to \cite{angiulia2023reinforcement} if readers are interested.

We start by reviewing the formulation of stochastic control problems in the infinite time horizon with continuous state space. Let $(\Omega, \mathcal{F}, \mathbb{P}) $ be a probability space accompanied with filtration $\{\mathcal{F}_t\}_{t\geq0} $  generated by a standard $n$-dimensional Brownian motion $B=\{B_t\}_{t\geq 0}$. For any time $t$, one has the state $X_t\in \mathcal{X}\subseteq\Rm^d$ following the McKean–Vlasov dynamics (i.e., distribution-dependent dynamics) and the Markovian control $\alpha_t=\alpha(X_t): \mathcal{X}\to \mathcal{A}\subseteq \Rm^k$. Given Borel-measurable functions
\begin{equation}
    b: \mathcal{X} \times  \mathcal{A}\times \mathcal{P}(\mathcal{X})\to \Rm^d,\quad \sigma:  \mathcal{X} \times  \mathcal{A}\times \mathcal{P}(\mathcal{X})\to \Rm^{d\times n}
\end{equation}
that satisfy necessary conditions for the well-posedness (see Section~\ref{sec:sdecontrol} for details). The stochastic control problem is that an agent controls her state $X$ via a sequence of actions (policy) $\alpha$ with the goal of minimizing the expected discounted cost
\begin{equation}\label{eqn:sc}
\begin{aligned}
			\inf_{\alpha}	J^{\mu  } ( \alpha )  &= 
			\inf_{\alpha} \mathbb{E} \left[ \int_{0}^{\infty} e^{-\gamma t} f(X_t , \alpha_t, \mu_t ) dt\right], \\
			\text{s.t. } & dX_{t} = b(X_t,\alpha_t,\mu_t) dt+ \sigma(X_t,\alpha_t,\mu_t)  dB_t,\quad  \quad  X_0 \sim \mu_0,
		\end{aligned}
\end{equation}
with a discount factor $\gamma >0$ and the probability measure flow $\mu_t$ starting from $\mu_0=\mathcal{P}[X_0]$, i.e., $\mu_t$ is the law of $X_t$. For more general versions of stochastic control problems and associated theory, we refer readers to the book by \cite{carmona2018probabilistic}.

The above general formulation with stochastic differential equation (SDE) control is based on unbounded state space $\mathcal{X}$. To be closely connected with reinforcement learning with a bounded state space $\mathcal{X}$ and an action space $\mathcal{A}$, we would consider MFG and MFC problems on discrete state space in the asymptotic sense following \cite[Section 2.2]{angiuli2022unified}. In this setup, the control does not depend on time but only on the state, since the transition probability $p$ and the cost function $f$ only depend on the limiting distributions other than time. The SDE control is replaced by the transition probability $p$, and the discount prefactor $e^{-\gamma t}$ is replaced by $r^k$ for some $r\in(0,1)$.
% with $b$ depending linearly on $\alpha_t$ and $\sigma$ being a constant. (see theory in \cite{carmona2018probabilistic}). \jl{what does this sentence mean? should say something that why we only consider the constant diffusion case}
% %
% Depending on the goals that players want to achieve, we consider the two types of mean field problems, mean field game (MFG) problem and mean field control (MFC) problem for sufficiently large $N$ players.
\paragraph{Mean Field Game (MFG)}    
Solving a MFG problem is to find a Nash equilibrium
$ (\hat\alpha,  \hat\mu)$ in a non-cooperative game by following:
\begin{enumerate}
    \item Fix a  probability distribution $\hat{\mu} \in\mathcal{P}(\mathcal{X})$  and solve the standard stochastic control problem 
\begin{equation}\label{opt_prob-mfg}
		\begin{aligned}
			\inf_{\alpha}	J^{ \hat{\mu}  } ( \alpha )  &= 
			\inf_{\alpha} \mathbb{E} \left[ \sum_{k=0}^\infty r^k f(X_k^{\alpha, \hat{\mu}} , \alpha(X_k^{\alpha,\hat{\mu}}),  \hat\mu ) \right], \\
			\text{s.t. } & X_{k+1}^{\alpha,\hat{\mu}}\sim p(\cdot \mid  X_k^{\alpha,\hat{\mu}}, \alpha(X_k^{\alpha,\hat{\mu}}), \hat\mu),\quad  X_0^{\alpha, \hat{\mu}} \sim \mu_0 ,
		\end{aligned}
\end{equation}
\item Given the  optimal control $\hat{\alpha}$, find the fixed point $\hat{\mu}$ such that
\[\hat{\mu}= \lim_{k\to\infty}\mathcal{P}[X^{\hat{\alpha},\hat{\mu}}_k]. \]
\end{enumerate}
% \jl{is $\mu_0$ in fact $\mu$ above?}
% with $\mu$ satisfying the fixed condition  $\cL[x_t^{ \alpha}] =  \mu, t \geq 0$. \jl{we shall probably include this as a constraint in the above equation}
% $x_t^{\alpha}$ is the state at time $t$ and follows the dynamic \eqref{eqn:sde} with the control $ \alpha$.
% $\cL[\cdot]$ is the probability density function of the state.
% Note that
% \eqref{opt_prob-mfg} sets the density sequence $ \left\{\mu_t , t\geq 0 \right\}$ in \eqref{opt_prob} to be the same: $\mu_t = \hat \mu$ for all $t \geq 0$. 
% %
% In this case,
% $\mu$ is first fixed to solve the stochastic optimization problem \eqref{opt_prob-mfg},
% then given the optimal control $ \alpha$ solved from minimization~\eqref{opt_prob-mfg}, 
% we solve the fixed point $\mu$ satisfying $\cL[x_t^{ \alpha}] =  \mu  $ for all $t \geq 0$. 
% %stationary
% \jl{does MFG have to require that the distribution stays the same for all time?}
\paragraph{Mean Field Control (MFC)}
Different from MFG that has fixed $\mu$ in the first step, the population distribution $\mu_k = \mathcal{P}[X_k^{\alpha}]$ in MFC changes instantaneuously when $\alpha$ changes. The asymptotic version of the problem is thus written as 
    \begin{equation}\label{opt_prob-mfc}
		\begin{aligned}
\inf_{\alpha}	J( \alpha )  &= 
			\inf_{\alpha} \mathbb{E} \left[  \sum_{k=0}^\infty r^k  f(X_k^{\alpha} , \alpha_k,  \lim_{k\to\infty}\mathcal{P}[X_k^\alpha] ) \right], \\
			\text{s.t. } & X_{k+1}^\alpha\sim p(\cdot \mid  X_k^\alpha, \alpha(X_k^\alpha), \lim_{k\to\infty}\mathcal{P}[X_k^\alpha]),\quad  X_0^{\alpha} \sim \mu_0,
		\end{aligned}
    \end{equation}
so that the control $\alpha$ is independent of time, as $p$ and $f$ depend only on the limiting distribution (as $k \to \infty$).    
% We remark that $\mu_t= \cL[X_t^\alpha]$ satisfies the Fokker-Planck equation
%     \begin{equation}\label{eq:mu}
%         \frac{\partial }{\partial t }\mu_t = -\sum_{i=1}^d \frac{\d}{\d x_i}(b_i(x_t,\alpha_t)\mu_t) +\frac{1}{2}\sum_{i=1}^d\sum_{j=1}^d\frac{\d^2}{\d x_i \d x_j}\Big[\sum_{k=1}^n \sigma_{ik}(x_t,\alpha_t)\sigma_{jk}(x_t,\alpha_t)\mu_t\Big].
%     \end{equation}

We emphasize that the main difference between the two is that in MFG, the distribution $\mu$ is prescribed when the optimal control is solved (and hence the superscript $\mu$ in the notation), while in MFC, the distribution depends on the choice of $\alpha$, when the policy is optimized.

% \paragraph{A unified view}
% Both MFG and MFC are concerned with the control and the distribution (density).
% The difference is that in MFG, the distribution is fixed during the optimization for the optimal control and then find a distribution satisfying the fixed point condition, while in MFC, the distribution changes instantaneously when the control function changes, as shown in \eqref{eq:mu}.
% Therefore \cite{angiuli2022unified} connects MFG and MFC problems through Q-learning %refs
% and Borkar's two timescale approach \cite{BORKAR1997291, borkar2008stochastic}.

% \paragraph{Notations and Assumptions}
% \begin{itemize}
%     \item $\|Q\|_{\infty}:= \sup_{x\in\mathcal{X}, a\in\mathcal{A}} |Q(x,a)|$ 
%     \item $\|Q\|_{1,1}:= \sum_{x\in\mathcal{X}}\sum_{a\in\mathcal{A}} |Q(x,a)|$
%     \item $\|\mu\|_2^2 = \langle \mu, \mu \rangle = \sum_{x\in\mathcal{X}} \mu(x)^2$; $\langle \mu, \nu \rangle = \sum_{x\in\mathcal{X}} \mu(x) \nu(x)$.
%     \item $\|\cdot\|$ denotes the Euclidean distance.
% \end{itemize}

% for discrete state space and control space 

\section{Value functions and Q-functions}\label{sec:problem}
We first recall formulations of value functions and Q-functions in both continuous and discrete time. With these, we establish the sequence of approximations in Fig.~\ref{fig:roadmap} from $Q_h$ which satisfies the Bellman equation to $V$ which solves the HJB equation.

\subsection{Value functions}
We recall the classical continuous-time \textit{value functions} for mean field game (MFG) and mean field control (MFC) problems, respectively. The value function of the MFG, with any fixed population distribution $\mu \in \mathcal{P}(\mathcal{X})$, is written as
\begin{equation}\label{def:V-MFG}
    V_{\text{MFG}}^{\alpha,\mu}(x) = \mathbb{E} \left[ \int_{0}^{\infty} e^{-\gamma s} f(X_s^{\alpha,\mu} , \alpha_s, \mu ) ds \bigg\vert X_0 =x \right].
    %\quad V^{\mu}(T,x) = 0.
\end{equation}
On the other hand, the value function of the MFC, different from MFG, has population distribution $\mu_t\in\mathcal{P}(\mathcal{X}) $ changing over time depending on the control. For asymptotic MFC, it is defined as
\begin{equation}\label{def:V-MFC}
   V_{\text{MFC}}^{\alpha} (x) =   \mathbb{E}\left[ \int_0^{\infty}  e^{-\gamma s} f(X_s^\alpha, \alpha_s, \lim_{t\to\infty}\cL[X_t^\alpha]) ds \bigg\vert X_0 =x \right].
\end{equation}
For formulations (\ref{def:V-MFG}) or (\ref{def:V-MFC}), the dynamics of $X_t$ follows a Markov process with 
\begin{equation}
    X_t\sim p(\cdot \mid X_{t'}, \alpha(X_{t'}), \mu)\quad \text{for} ~t' < t,
\end{equation}
where $\mu$ is fixed for MFG and $\mu=  \lim_{t\to\infty}\cL[X_t]$ for MFC (recall we consider the asymptotic MFC in this work).
% where $G(x,\alpha):\mathcal{X}\times \mathcal{A} \to \mathcal{P}(\mathcal{X})\times \mathcal{P}(\mathcal{X})$ is a continuous nonlinear infinitesimal generator.

Analogously, if we consider the discrete MDP $(\mathcal{X},\mathcal{A}, e^{-\gamma h}, f_k)$ as a counterpart of the continuous-time MDP with time discretization $h$, and we use the notations  $X_k \equiv X_{k h}, \alpha_k \equiv \alpha_{kh}= \alpha(X_k), \mu_k\equiv \mu_{kh}, f_k \equiv f(X_k,\alpha_k,\mu_k)$, then given an admissible policy $\alpha$, the \textit{discrete  value functions} $V_h^{\alpha}$ for MFG has the form of
\begin{equation}\label{Vh_MFG}
     V_{h, \text{MFG}}^{\alpha,\mu}(x) :=  \mathbb{E} \left[h \sum_{k=0}^{\infty}  e^{-k\gamma h} f(X_k^{\alpha,\mu} , \alpha_k, \mu )  \bigg\vert X_0 =x  \right],
\end{equation}
with the state $X_k^{\alpha,\mu}$ changes by
\begin{equation}
    X_{k+1}^{\alpha,\mu} \sim p( \cdot \mid X_k^{\alpha,\mu}, \alpha_k, \mu ).
\end{equation}
Similarly, for MFC, we have the form
\begin{equation}\label{Vh_MFC}
     V_{h, \text{MFC}}^{\alpha} (x) =   
     \mathbb{E}\left[ h \sum_{k=0}^\infty e^{-k\gamma h} f(X_{k}^\alpha, \alpha_{k}, \lim_{k\to\infty}\cL[X_k^\alpha])  \bigg\vert X_0 =x \right],
\end{equation}
with the state $X_k^\alpha$ changes by
\begin{equation}
    X_{k+1}^\alpha \sim p( \cdot \mid X_k^\alpha, \alpha_k, \lim_{k\to\infty}\cL[X_k^\alpha] ).
\end{equation}
For (\ref{def:V-MFG}) and (\ref{def:V-MFC}),  we can derive the optimal value functions by optimizing over policies $\alpha$:
\begin{equation}
    V_{\text{MFG}}^{\mu}(x) = \inf_{\alpha}V_{\text{MFG}}^{\mu,\alpha}(x),\quad  V_{\text{MFC}}(x) = \inf_{\alpha}V_{\text{MFC}}^{\alpha}(x),
\end{equation}
Similarly, for (\ref{Vh_MFG}) and (\ref{Vh_MFC}), the discrete optimal value functions are defined as
\begin{equation}
    V_{h, \text{MFG}}^{\mu}(x) = \inf_{\alpha_h}V_{h, \text{MFG}}^{\mu,\alpha}(x),\quad  V_{h, \text{MFC}}(x) = \inf_{\alpha_h}V_{h, \text{MFC}}^{\alpha}(x).
\end{equation}
In addition, we introduce the assumption on the cost function $f$ that will be used throughout the paper.
\begin{assumption}\label{assump:f}
We assume that the cost function $f: \mathcal{X} \times \mathcal {A}  \times \mathcal{P}(\mathcal{X}) \rightarrow \mathbb{R}$ is bounded and Lipschitz continuous in $\mu$, in the sense that there exists a constant $ L_{\mu}>0$ such that for every $(x,a)\in \mathcal{X} \times \mathcal {A}$,
\begin{equation}
\begin{aligned}
    \big|f(x,a,\mu_1) - f(x,a,\mu_2)\big| \leq  L_{\mu} \|\mu_1-\mu_2\|_{\text{TV}}
     \quad  \text{for any }\mu_1, \mu_2 \in \mathcal{P}(\mathcal{X}).
\end{aligned}
\end{equation}
\end{assumption}

\subsection{HJB equations with SDE controls}\label{sec:sdecontrol}
While for most of this work, we consider discrete state space, we study in this section the continuous state space analog, where the state dynamics is given by stochastic differential equations (controlled diffusion)
\begin{equation}
    dX_{t} = b(X_t,\alpha_t, \mu_t) dt+ \sigma(X_t,\alpha_t, \mu_t)  dB_t.
\end{equation}
We use this setup to review the familiar Hamilton-Jacobi-Bellman equations, which would shed light on the difference between MFG and MFC in the continuous-time solution viewpoint. Furthermore, our numerical experiments in Section~\ref{sec:numeric} are based on discretizations of the SDE.

For continuous state space models, we require some additional assumptions to ensure the wellposedness of the problem. 

\begin{assumption}\label{assump:f_extra}
Given an unbounded state space $\mathcal{X}$, we assume that the cost function $f: \mathcal{X} \times \mathcal {A}  \times \mathcal{P}(\mathcal{X}) \rightarrow \mathbb{R}$ is bounded and measurable. For any fixed $\mu \in\mathcal{P}(\mathcal{X})$,  $f$ is Lipschitz continuous in $x,a$, in the sense that there exist constants $L_x, L_{\alpha}>0$ such that 
\begin{equation}
\begin{aligned}
    \big|f(x_1,\alpha_1,\mu) - f(x_2,\alpha_2,\mu)\big| \leq  L_{x} \|x_1-x_2\| +L_{\alpha}\|\alpha_1-\alpha_2\|
     \quad \text{for any } x_1, x_2\in\mathcal{X},~a_1, a_2\in \mathcal{A}.
\end{aligned}
\end{equation}
\end{assumption}
\begin{assumption}\label{assump:b-sigma}
We assume that for any $(x,a, \mu) \in \mathcal{X} \times \mathcal{A}\times \mathcal{P}(\mathcal{X})$, both $b(x, a, \mu)$ and $\sigma(x,a,\mu)$ are measurable, bounded, and Lipschitz in $x,a$, which means that  there exist  constants $K_x, K_\alpha>0$, and for every $\mu\in \mathcal{P}(\mathcal{X})$, it holds uniformly that
\begin{equation}
    \begin{aligned}
        &\|b(x',a',\mu)-b(x,a,\mu)\|\leq K_x\|x'-x\|+K_\alpha\|a'-a\|,\\
        &\|\sigma(x',a',\mu)-\sigma(x,a,\mu)\|\leq K_x\|x'- x\|+K_\alpha\|a'-a\|.
    \end{aligned}
\end{equation}
Moreover, both $b(x, a,\mu)$ and $\sigma(x,a,\mu)$ are differentiable in $x$ and $a$.
\end{assumption}

Given the optimal value functions $V$, one can obtain the  HJB equations for MFG and MFC by following the derivations in Chapter 3 and Chapter 4 of \cite{Bensoussan2013} respectively. For simplicity, we give the statement with constant $\sigma(x,a)\equiv \sigma>0$.

\begin{definition}
    We say $f(x,a,\mu)$ is differentiable in $\mu$ if the first variation
    \begin{equation}
        \frac{d}{d\epsilon}f(x,a,\mu+\epsilon m) \Big|_{\eps=0}:=\int_{\mathcal{X}}\frac{\delta f(x,a,\mu)}{\delta \mu}(\xi) \, m(d \xi) 
    \end{equation}
    exists, for any $m\in\mathcal{P}(\mathcal{X})$.
\end{definition}
\begin{theorem}[\cite{Bensoussan2013}]\label{thm:V-HJB}
Under Assumptions \ref{assump:f}, \ref{assump:f_extra}, and \ref{assump:b-sigma}, with the Hamiltonian
\begin{equation}\label{eqn:hamiltonian}
    H(x, \mu, q) := \inf_{\alpha} \left\{ q \cdot b(x,\alpha,\mu)   + f(x, \alpha, \mu)\right\},
\end{equation}
the optimal value function
$ V_{\text{MFG}}^\mu$ for asymptotic MFG satisfies the HJB equation
\begin{equation}\label{eqn:MFG-HJB}
    -\gamma V^{\mu}(x)+\frac{\sigma^2 \Tr{\nabla^2 V^{\mu}(x)}}{2}+ H(x,\mu, \nabla V^{\mu}(x))=0.
\end{equation}
% coupled with the Fokker-Planck equation
% \begin{equation}\label{eqn:fp}
%         \frac{\d }{\d t }\mu = -\sum_{i=1}^d \frac{\d}{\d x_i}(b_i(x,-\nabla V^{\mu})\mu) +\frac{\sigma^2}{2}\Delta \mu,\quad \mu(x,0)=\mu_0(x).
% \end{equation}

On the other hand, the optimal value function
$ V_{\text{MFC}}$ for asymptotic MFC satisfies the HJB equation
\begin{equation}\label{eqn:MFC-HJB}
    -\gamma V(x) + \frac{\sigma^2 \Tr{\nabla^2 V(x)}}{2}+ H(x,\mu, \nabla V(x)) + \int_{\mathcal{X}}\frac{\delta  H(y, \mu, \nabla V(y))}{\delta \mu}  (x) \, \mu (dy) =0, 
\end{equation}
coupled with $\mu$ solving the stationary Fokker-Planck equation
\begin{equation}\label{eqn:fp}
        -\sum_{i=1}^d \frac{\d}{\d x_i}(b_i(x, \hat{\alpha}, \mu )\mu) +\frac{\sigma^2}{2}\Delta \mu= 0,
\end{equation}
where $\hat{\alpha}$ is the optimal control for the Lagrangian in (\ref{eqn:hamiltonian}).
\end{theorem}
From the above HJB equations, it is straightforward to see that, in general $V_{\text{MFG}}^\mu$ and $ V_{\text{MFC}}$ are different solutions, as in the case of MFC the HJB equation has an additional term due to the coupling with $\mu$. In next few sections, we will study how the two-timescale Q-learning algorithm converges to these different value functions. 

% The main motivation of our work is to study how the two-timescale algorithms (to be recalled in the next section) converge to these different value functions. 

\medskip 

The following results state that given  any policy $\alpha$, the discrete value function is close to the continuous-time value function for sufficiently small $h$, which implies similar approximation results  for optimal value functions.
\begin{theorem}\label{thm:conv_h}(Informal version of Theorem \ref{thm:conv_formal})
%For MFG and MFC, show convergence of $V_h $ to $V $ when $h\to 0$.
%Assume that $f: \mathcal{X} \times \mathcal {A}  \times \mathcal{P}_2(\mathcal{X}) \rightarrow \mathbb{R}$ is bounded, and
%that $x \rightarrow f(x, \alpha(x), \mu)$ is $L_r$-Lipschitz continuous, then 
Under appropriate assumptions, under an given policy $\alpha$,
for all $x \in {\mathcal{X}}$, one has approximations
\begin{equation}
\label{eq:conv-value}
\lim_{h\to 0}V_{h, \text{MFG}}^{\mu,\alpha}(x) = V_{\text{MFG}}^{\mu,\alpha}(x),
\quad
\lim_{h\to 0}V_{h, \text{MFC}}^{\alpha} (x)  = V_{\text{MFC}}^{\alpha}(x).
\end{equation}
\end{theorem}
The idea of the proof is standard by extending Lemma 1 in \cite{tallec19a} to a stochastic version with additional assumptions. We defer the formal statement with convergence rates,  as well as proof details to Appendix~\ref{sec:appd}.%{appe:conv_h}.

\begin{corollary}
By Theorem~\ref{thm:conv_h} and taking the optimal control $\hat{\alpha}$, we have that the approximations for the optimal value functions
\begin{equation}
\label{eq:conv-optimal-value}
\lim_{h\to 0}V_{h, \text{MFG}}^{\mu}(x) = V_{\text{MFG}}^{\mu}(x),
\quad
\lim_{h\to 0}V_{h, \text{MFC}}(x)  = V_{\text{MFC}}(x).
\end{equation}
\end{corollary}

\subsection{Q-functions}
Following the context of asymptotic MFG (\ref{opt_prob-mfg}) introduced in \cite{angiuli2022unified},  the discrete time Q-function (i.e., state-action value function) and optimal Q-function are defined as
\begin{equation}\label{MFG_Q_h}
\begin{aligned}
&Q_{h,\text{MFG}}^{\alpha,\mu} (x,a)
         :=  \mathbb{E}\left[ h \sum_{k=0}^\infty e^{-k\gamma h} f(X_{k}^{\alpha,\mu}, \alpha_{k}, \mu ) \bigg\vert X_0 =x, \alpha_0 =a \right],\\
&Q^{\mu}_{h,\text{MFG}}(x,a) :
         = \inf_{ \alpha }  
         Q_{h,\text{MFG}}^{\alpha,\mu} (x,a).
\end{aligned}
\end{equation}
Optimizing over initial actions, we have that
$V_{h,\text{MFG}}^{\mu}(x) = \inf_{a} Q_{h,\text{MFG}}^{\mu}(x,a)$ for any  $x\in \mathcal{X}$. Moreover, as $\mu$ is fixed, by  \cite[Equation (3.20)]{sutton2018reinforcement}, $Q_{h,\text{MFG}}^{\mu}$ satisfies the Bellman equation
\begin{equation}\label{eqn:MFG_Bellman}
    Q_{h,\text{MFG}}^{\mu}(x,a) = h f(x,a,\mu) + e^{-\gamma h} \sum_{x'\in\mathcal{X}} p(x'\mid x,a,\mu) \inf_{a'}Q_{h,\text{MFG}}^{\mu}(x',a').
\end{equation}

On the other hand, the case of MFC is more complicated.  In the context of asymptotic MFC (\ref{opt_prob-mfc}), considering $\mu^\alpha$ to be the limiting distribution of the process $X_t^{\alpha}$ for an admissible policy $\alpha$, the discrete time modified Q-function introduced in \cite{angiuli2022unified} is defined as
\begin{equation}\label{MFC_Q_h}
         Q_{h, \text{MFC}}^{\alpha} (x,a)
         := h f(x,a, \mu^{\tilde{\alpha}})+\mathbb{E}\left[ h \sum_{k=1}^\infty e^{-k\gamma h} f(X_k^\alpha, \alpha_k,  \mu^{\alpha} ) \bigg\vert X_0 =x, \alpha_0 =a \right],
\end{equation} 
where
\begin{equation}\label{policy}
\mu^{\alpha}=\lim_{k\to\infty}\cL[X_k^{\alpha}],\quad \quad \tilde{\alpha}(s)=  \begin{cases}
\alpha(s),\quad \text{if}~~ s\neq x\\
a, \quad \quad ~ \text{if}~~ s=x.
\end{cases}
\end{equation}
% This is also called an integrated Q-function, which is different from Q-function for RL of single agent.
%
We mention that $\tilde \alpha$ is devised in such a form (\ref{policy}) in order to achieve policy improvement  \cite[Theorem 4 in Appendix C]{angiuli2022unified}. Then the optimal  Q-function $Q_{h, \text{MFC}}$ is
\begin{equation}
    Q_{h, \text{MFC}}(x,a)
         = \inf_{ \alpha }  
         Q_{h, \text{MFC}}^{\alpha} (x,a),
\end{equation}
which satisfies the Bellman equation
\begin{equation}\label{modified_Q_Bellman}
    Q_{h,\text{MFC}}(x,a)= h f(x,a,\tilde{\mu}^* )
		+ e^{-\gamma h} \sum_{x'\in\mathcal{X}}p(x'\mid x,a,\tilde{\mu}^*)\inf_{a'} Q_{h,\text{MFC}}(x',a'),
\end{equation}
for each $(x,a)\in\mathcal{X}\times \mathcal{A}$.
The optimal control $\alpha^*(x) = \argmin_{a\in\mathcal{A}} Q_{h,\text{MFC}}(x,a)$, and the control $\tilde{\alpha}^*$ is also defined as in (\ref{policy}). The modified population distribution $\tilde{\mu}^*$ is based on $\tilde{\alpha}^*$ in the sense that $\tilde{\mu}^*= \mu^{\tilde{\alpha}^*}$. Let us summarize the discrete time Q-function results for asymptotic MFC as follows.
\begin{thm}[\cite{angiuli2022unified}, Appendix C]
The Bellman equation for the discrete time Q-function $Q_{h, \text{MFC}}^{\alpha}$ is 
\begin{equation}\label{dec00}
    Q_{h, \text{MFC}}^{\alpha} (x,a) = h f(x,a, \mu^{\tilde\alpha}) + e^{-\gamma h} \mathbb{E}\left[Q_{h, \text{MFC}}^{\alpha}(X_1, \alpha(X_1))|X_0=x, \alpha_0= a\right]
\end{equation}
with $\tilde\alpha$ defined as in \eqref{policy}. 
Moreover, for any $x\in\mathcal{X}$, the value function is equivalent to the Q-function with the policy $\alpha$ in the form of
\begin{equation}\label{dec02}
     V_{h, \text{MFC}}^{\alpha}(x) = Q_{h, \text{MFC}}^{\alpha}(x, \alpha(x)).
\end{equation}
The optimal Q-function $Q_{h, \text{MFC}}(x,a)
         = \inf_{ \alpha}  
         Q_{h, \text{MFC}}^{\alpha} (x,a)$
satisfies the Bellman equation
\begin{equation}
    Q_{h,\text{MFC}}(x,a)= h f(x,a,\tilde{\mu}^* )
		+ e^{-\gamma h} \sum_{x'\in\mathcal{X}}p(x'\mid x,a,\tilde{\mu}^*)\inf_{a'} Q_{h, \text{MFC}}(x',a'),
\end{equation}
with $\tilde{\mu}^*= \mu^{\tilde{\alpha}^*}$ and $\tilde{\alpha}^*$ being the modified control (\ref{policy}) of the optimal control $\alpha^*$.
%  And the function $ Q_{h}$ satisfies the  Bellman equation
% \begin{equation} \label{eq:MFC_Bellman}
% 		Q_{h}(x,a) =  h f(x,a,\cL(x) )
% 		+ e^{-\gamma h} \inf_{ a' } \mathbb{E} \left[ Q_{h}( \xi(x, a, h), a'  ) \right],
% \end{equation}
% where $\xi(x,a,h)$ denotes the solution of the SDE at time $t = h$ with initial state $x_0 = x$ and constant action $\alpha_0 \equiv a $ for $t \in [0, h)$.
\end{thm}
% Here, $\xi$ is the state evolution map such that
% \begin{equation}\label{eqn:evolution}
%     x_{k+1} = \xi(x_k,\alpha_k, h),
% \end{equation}
% and it is consistent with the discrete-time version of the SDE (\ref{eqn:sde}), which is 
% \begin{equation}
%     x_{k+1} = x_k+ b(x_k,\alpha_k) h + \sigma \sqrt{h}~ W,\quad W\sim \mathcal{N}(0,I).
% \end{equation}

\begin{proof}
All proofs can be found in \cite[Appendix C]{angiuli2022unified}. We review the proofs of the first two statements here and delegate the last one to the reference.

\cite[Appendix C, Theorem 3]{angiuli2022unified} : By the tower property, the definition of (\ref{MFC_Q_h}) gives 
\begin{equation*}
    \begin{aligned}
 &Q_{h, \text{MFC}}^{\alpha} (x,a)=  h f(x,a, \mu^{\tilde{\alpha}})+e^{-\gamma h}\mathbb{E}\left[\mathbb{E} \left[ h \sum_{k=1}^\infty e^{-(k-1)\gamma h} f(X_k, \alpha_k, \mu^{\alpha})\bigg\vert X_1\right] \bigg\vert X_0 =x, \alpha_0 =a \right] \\
 &=h f(x,a, \mu^{\tilde{\alpha}})\\
 &\quad +e^{-\gamma h}\mathbb{E}\left[ hf(X_1, \alpha(X_1), \mu^{\alpha})+ e^{-\gamma h}\mathbb{E} \left[ h \sum_{k=2}^\infty e^{-(k-2)\gamma h} f(X_k, \alpha_k, \mu^{\alpha})\bigg\vert X_1\right]\bigg\vert X_0 =x, \alpha_0 =a \right] \\
 &= h f(x,a, \mu^{\tilde{\alpha}}) + e^{-\gamma h} \mathbb{E}\left[Q_{h, \text{MFC}}^{\alpha}(X_1, \alpha(X_1))|X_0=x, \alpha_0 =a\right].
\end{aligned}
\end{equation*}
% By the definition of the value function $V_h(x)$ in MDP, the infimum is taken over all admissible policies $\alpha$, and for $Q(x,a)$, the infimum is taken over all admissible policies $\alpha$ with a further restriction that $\alpha_0 = a$. Note that

\cite[Appendix C, Lemma 3]{angiuli2022unified}: By the form of modified control \eqref{policy}, the discrete value function can be written as
\begin{equation*}
    \begin{aligned}
        V_{h, \text{MFC}}^{\alpha} (x)&= h f(x,\alpha(x),\mu^{\tilde \alpha}) + 
     \mathbb{E}\left[ h \sum_{k=1}^\infty e^{-k\gamma h} f(X_{k}, \alpha_{k}, \mu^{\alpha})  \bigg\vert X_0 =x, \alpha_0=\alpha(x)\right]\\
     &=Q_{h, \text{MFC}}^{\alpha} (x,\alpha(x)).
    \end{aligned}
\end{equation*}

\end{proof}

\subsection{Approximation results for value functions}
The Q-function is ill-posed for the continuous-time MDP \citep{tallec19a}, as it becomes independent of actions when $h\to 0$. However, one can measure the difference between discrete value functions and Q-functions in terms of $h$, when the control $\alpha$ is fixed.
% one can  as
% \begin{equation}
%     Q_{h,MFG}^{\alpha}(x,a) = V_{h,MFG}^{\alpha}(x)+O(h),\quad \forall (x,a).
% \end{equation}
% Given the formulations (\ref{dec00}) and (\ref{dec02}), we can derive a similar difference estimate for 
Such a distance measure result can be found in \cite[Theorem 2]{tallec19a} for MFG problems when the state is driven by the deterministic differential equation. Here, we provide similar results for both MFG and MFC under the McKean-Vlasov dynamics control, based on formulations (\ref{eqn:MFG_Bellman}), (\ref{dec00}),  and (\ref{dec02}).
    
\begin{thm}[Difference between $Q_{h}^{\alpha}$ and $V_{h}^{\alpha}$]\label{thm:dist_MFG}
Let $\one_x$ be the unit point mass probability distribution over $\mathcal{X}$. Consider the infinitesimal generator $G$ given by 
\begin{equation}\label{discrete_SDE}
    G(\cdot \mid x,a,\mu) = \lim_{h\to 0}\frac{p_h(\cdot \mid x,a,\mu)-\one_x}{h}
\end{equation}
being uniformly bounded for all $(x',x,a,\mu)\in \mathcal{X}\times\mathcal{X}\times \mathcal{A}\times \mathcal{P}(\mathcal{X})$, and $p_h(\cdot \mid x,a,\mu)$ is the one-step transition probability with respect to the time step $h$. If $f$ is uniformly bounded over $\mathcal{X}\times\mathcal{A}\times\mathcal{P}(\mathcal{X})$, following the one-step McKean-Vlasov dynamics $x'\sim p_h(\cdot \mid x,a,\mu)$,  we have that
\begin{equation}
\begin{aligned}
     &Q_{h,\text{MFG}}^{\mu,\alpha}(x,a)  =  V_{h,\text{MFG}}^{\mu,\alpha}(x) + O(h),\\
       &Q_{h,\text{MFC}}^{\alpha}(x,a)  = V_{h,\text{MFC}}^{\alpha}(x)  + O(h),
       \end{aligned}
\end{equation}
with sufficiently small $h>0$, for every $(x,a)\in \mathcal{X}\times\mathcal{A}$. 
\end{thm}
\begin{proof}
As $f$ is uniformly bounded, $V$ is also uniformly bounded over $\mathcal{X}$ by its formulations.
%{\color{blue}
%For MFG, the Bellman equation gives that
%\begin{equation*}
%\begin{aligned}
%    Q_{h,MFG}^{\mu, \alpha}(x,a) &=  h f(x,a,\mu )
%		+ e^{-\gamma h} \E\left[Q_{h, MFG}^{\mu, \alpha}(x_1, \alpha(x_1))|x_0=x, \alpha_0= a\right]\\
%  &=h f(x,a,\mu ) + (1-\gamma h) \sum_{x'\in\mathcal{X}} p(x'\mid x,a,\mu) V_{h,MFG}^{\mu,\alpha}(x')+O(h^2).
%  \end{aligned}
%\end{equation*}
%By Ito's formula, we have that
%\begin{equation*}
%\begin{aligned}
%     V_{h, MFG}^{\mu,\alpha}(x') &=  V_{h, MFG}^{\mu,\alpha}(x) + \Big(b(x,a)\cdot \nabla_xV^{\mu,\alpha}_{h,MFG}(x) +\frac{\sigma^2}{2}\Delta_xV^{\mu,\alpha}_{h,MFG}(x)\Big)h
%    + \sigma \sqrt{h}\nabla_xV^{\mu,\alpha}_{h,MFG}(x)\cdot W
%    . 
%    \end{aligned}
%\end{equation*}
%Therefore, replacing $ V_{h, MFG}^{\mu,\alpha}(x')$ in the Bellman equation an taking the expectation with respect to the Gaussian noise on both sides, we get
%\begin{equation*}
%\E\big[Q^{\mu,\alpha}_{h,MFG}(x,a) \big] = h f(x,a,\mu )
%		+ (1-\gamma h) \big(\E\big[V^{\mu,\alpha}_{h,MFG}(x)\big]+O(h)\big)+O(h^2) = \E\big[V^{\mu,\alpha}_{h,MFG}(x)\big] + O(h).
%\end{equation*}
%} 

%blue to red: $x_0=x$ is a realization of random state at time 0, and $x'$ is random?
%{\color{red}

For MFG, the Bellman equation gives that
\begin{equation*}
\begin{aligned}
    Q_{h, \text{MFG}}^{\mu, \alpha}(x,a) &=  h f(x,a,\mu )
		+ e^{-\gamma h} \mathbb{E}\left[Q_{h, \text{MFG}}^{\mu, \alpha}(X_1, \alpha(X_1))|X_0=x, \alpha_0= a\right]\\
  &=h f(x,a,\mu ) + (1-\gamma h) \mathbb{E}\left[ V_{h,\text{MFG}}^{\mu,\alpha}(x') \right] +O(h^2),
  \end{aligned}
\end{equation*}
with $X_1 =x'$ and sufficiently small $h$. Note that
\begin{equation}
\begin{aligned}
    \mathbb{E}\left[ V_{h,\text{MFG}}^{\mu,\alpha}(x') \right]&= \sum_{x'} V_{h,\text{MFG}}^{\mu,\alpha}(x')p_h(x'\mid x,a,\mu) \\
    &=V_{h,\text{MFG}}^{\mu,\alpha}(x)+\sum_{x'} V_{h,\text{MFG}}^{\mu,\alpha}(x')  (p_h(x'\mid x,a,\mu)-\one_x)\\
    &= V_{h,\text{MFG}}^{\mu,\alpha}(x)+\sum_{x'} V_{h,\text{MFG}}^{\mu,\alpha}(x') G(x'\mid x,a,\mu) h +\lito(h)=V_{h,\text{MFG}}^{\mu,\alpha}(x)+O(h)
    \end{aligned}
\end{equation}
as the generator $G$ is uniformly bounded and $\mathcal{X}$ is finite.
% By It\^{o}'s formula and \eqref{discrete_SDE}, we have that
% \begin{equation*}
% \begin{aligned}
%      V_{h, \text{MFG}}^{\mu,\alpha}(x') &=  V_{h, \text{MFG}}^{\mu,\alpha}(x) + \Big(b(x,a)\cdot \nabla_xV^{\mu,\alpha}_{h, \text{MFG}}(x) +\frac{\sigma^2}{2}\Delta_xV^{\mu,\alpha}_{h, \text{MFG}}(x)\Big)h
%     + \sigma \sqrt{h}\nabla_xV^{\mu,\alpha}_{h, \text{MFG}}(x)\cdot B +O(h^{3/2}).
%     \end{aligned}
% \end{equation*}
Therefore, by replacing $ \mathbb{E}\left[V_{h, MFG}^{\mu,\alpha}(x') \right]$ in the Bellman equation, we get
\begin{equation*}
Q^{\mu,\alpha}_{h,\text{MFG}}(x,a)  = h f(x,a,\mu )
		+ (1-\gamma h) \big( V^{\mu,\alpha}_{h, \text{MFG}}(x)+O(h)\big)+O(h^2) = V^{\mu,\alpha}_{h,\text{MFG}}(x)+ O(h),
\end{equation*}
for sufficiently small $h$. The estimate for MFC is similar by just replacing $\mu$ by the limiting distribution $\mu^{\alpha}= \lim_{k\to \infty}\mathcal{P}[X_k^{\alpha}]$ under an admissible policy $\alpha$. The Bellman equation (\ref{dec00}) combined with (\ref{dec02}) gives 
\begin{equation}\label{feb06}
\begin{aligned}
     Q_{h, \text{MFC}}^{\alpha} (x,a) &= h f(x,a, \mu^{\tilde\alpha}) + e^{-\gamma h} \mathbb{E}\left[V_{h, \text{MFC}}^{\alpha}(x')|X_0=x, \alpha_0= a\right]\\
      &=h f(x,a,\mu^{\tilde\alpha}) + (1-\gamma h) \mathbb{E}\left[ V_{h,\text{MFC}}^{\alpha}(x') \right] +O(h^2),
      \end{aligned}
\end{equation}
for sufficiently small $h$. Since
\begin{equation}
\begin{aligned}
    \mathbb{E}\left[ V_{h,\text{MFC}}^{\mu,\alpha}(x') \right]&= \sum_{x'} V_{h,\text{MFC}}^{\mu,\alpha}(x')p_h(x'\mid x,a,\mu^\alpha)\\
    &= V_{h,\text{MFC}}^{\mu,\alpha}(x)+\sum_{x'} V_{h,\text{MFC}}^{\mu,\alpha}(x')  G(x'\mid x,a,\mu^\alpha) h +\lito(h)=V_{h,\text{MFC}}^{\mu,\alpha}(x)+O(h),
    \end{aligned}
\end{equation}
then the Bellman equation gives that
\begin{equation*}
  Q^{\alpha}_{h,\text{MFC}}(x,a)  = h f(x,a,\mu^{\tilde\alpha} )
		+ (1-\gamma h) \big( V^{\alpha}_{h,\text{MFC}} (x)+O(h)\big)+O(h^2) = V^{\alpha}_{h,\text{MFC}}(x) + O(h)
\end{equation*}
for sufficiently small $h$.
\end{proof}

%taking the infimum wrt to control
\begin{corollary}
For discrete optimal value functions and Q-functions, by taking the infimum over all admissible policies $\alpha$, we have that
\begin{equation}
     Q_{h,\text{MFG}}^{\mu}(x,a)  =  V_{h,\text{MFG}}^{\mu}(x) + O(h),
\end{equation}
and
\begin{equation}
     Q_{h,\text{MFC}}(x,a)  = V_{h,\text{MFC}}(x)  + O(h),
\end{equation}
for sufficiently small $h>0$, for every $(x,a)\in \mathcal{X}\times \mathcal{A}$.
\end{corollary}

\section{Two-timescale Q-learning algorithm}\label{sec:alg}
Given the discrete time Q-functions and the associated Bellman equations formulated in the previous section, we first recall the two-timescale Q-learning algorithm introduced in \cite{angiuli2022unified}. We will take the continuous-time approximation of the algorithm, and analyze its different fixed point solutions in both MFG and MFC regimes. Then we construct a toy one-dimensional example in which explicit fixed point solutions can be obtained under different learning rates ratios. In the end we validate our findings for the example by numerical simulations.

This iterative procedure, starting from some initial guess $(Q_0, \mu_0)$, updates both variables at each iteration $k$ with different learning rates, $\rho_k^{Q}>0$ and $\rho_k^\mu>0$:
\begin{equation}\label{eqn:alg}
    \begin{aligned}
    \mu_{k+1} &= \mu_k+\rho_k^{\mu} \, \mathcal{P}(Q_k, \mu_k),\\
    Q_{k+1} &= Q_k + \rho_k^{Q} \, \mathcal{T}(Q_k, \mu_k),
    \end{aligned}
\end{equation}
with operators 
\begin{equation}
    \begin{aligned}
       \mathcal{P}(Q, \mu)(x) &= (\mu P^{Q,\mu})(x)-\mu(x),~~\text{for}~x\in\mathcal{X},\\
       (\mu P^{Q,\mu})(x) & = \sum_{x_0} \mu(x_0) P^{Q,\mu}(x_0,x), 
       \quad P^{Q,\mu}(x,x') = p(x'\mid x,\argmin_a Q(x,a),\mu),
       \end{aligned}
\end{equation}
and
\begin{equation}
      \begin{aligned}
   \mathcal{T}(Q, \mu)(x,a)&= hf(x,a,\mu)+e^{-\gamma h} \sum_{x'}   p(x'\mid x, a,\mu)\min_{a'} Q(x',a')-Q(x,a),~~\text{for}~(x,a)\in \mathcal{X}\times \mathcal{A}.
    \end{aligned}
\end{equation}
When (\ref{eqn:alg}) converges to a stationary point $(Q_h^*, \tilde{\mu}^*)$, this stationary point satisfies a fixed-point equation (cf.~the Bellman equation 
\eqref{modified_Q_Bellman}): for all $(x,a)\in\mathcal{X}\times\mathcal{A}$,
\begin{equation}\label{eqn:soln}
\begin{aligned}
    \tilde{\mu}^*(x) &= \tilde{\mu}^* P^{Q^*, \tilde{\mu}^*}(x),\\
    Q_h^*(x,a) &=hf(x,a,\tilde{\mu}^*)+e^{-\gamma h} \sum_{x'}   p(x'\mid x, a,\tilde{\mu}^*)\min_{a'} Q_h^*(x',a').
\end{aligned}
\end{equation}
This two-timescale approach can converge to different limiting points by simply tuning two learning rates. Following the idea of \cite{borkar1997stochastic, borkar2008stochastic}, if $\eps := \rho^{\mu}/\rho^{Q} \ll 1$, the numerical updates \eqref{eqn:alg} can be approximated by a system of two-timescale ordinary differential equations (ODEs) 
\begin{equation}
    \begin{aligned}
    \frac{d}{d t} \mu_t &= \mathcal{P}(Q_t, \mu_t),\\
    \frac{d}{d t} Q_t  &= \frac{1}{\eps} \mathcal{T}(Q_t, \mu_t).
    \end{aligned}
\end{equation}
As $\eps \to 0$, $\mu_t$ changes much slower than $Q_t$. So for the consideration of dynamics of $Q_t$, we can treat $\mu_t$ as frozen $\mu_t \equiv\mu$, and then the stable equilibrium point satisfies $\mathcal{T}(Q_{\mu}, \mu)=0$. We use the notation $Q_{\mu}$ as the equilibrium point depending on $\mu$. Moreover, by assuming $Q_\mu$ is Lipschitz continuous in $\mu$ (a verification can be found in estimates in Section \ref{sec:conv}, such like (\ref{mar26_lip}) using the Lipschitz continuity assumptions of $f, p$). 
Given the pair $(Q_{\mu},\mu)$, we then consider to solve 
$  \frac{d}{d t}  \mu_t = \mathcal{P}(Q_{\mu_t},\mu_t)$,  
which gives the eventual solution that satisfies $\mathcal{P}(Q_{\mu_\infty},\mu_\infty)=0$. From the stable equilibrium solution $(Q_{\mu_\infty},\mu_\infty)$, 
the derived $\mu_{\infty}$ and optimal control $\hat\alpha$ that minimizes $Q_{\mu_\infty}$ form a Nash equilibrium of MFG. Therefore, we call $\rho^{\mu}\ll\rho^{Q}$ the MFG regime.

On the other hand, when $\rho^Q\ll \rho^{\mu}$, we take the ratio $\rho^{Q}/\rho^{\mu}$ to be of order $\eps\ll 1$, then by a similar strategy, we have the approximated system of ODEs
\begin{equation}
    \begin{aligned}
    \frac{d}{d t} \mu_t &= \frac{1}{\eps} \mathcal{P}(Q_t, \mu_t),\\
    \frac{d}{d t} Q_t  &= \mathcal{T}(Q_t, \mu_t).
    \end{aligned}
\end{equation}
As $\eps \to 0$,  $Q_t$ changes much slower than $\mu_t$. We can thus freeze $Q_t\equiv Q$ when the dynamics of $\mu_t$ is concerned, and it leads to the stationary point $\tilde{\mu}_Q$ satisfying $\mathcal{P}(Q, \mu_Q)=0$, where $\mu_Q$ is the asymptotic distribution of a population in which every agent uses the
control $\alpha(x) =\argmin_a Q(x,a)$. We may assume $\mu_Q$ is  Lipschitz continuous in $Q$ (a verification can be found in  Section \ref{sec:conv}, Lemma \ref{lem:mu_and_Q}), and replace $\mu_Q$ by $\tilde{\mu}_Q$ defined with modified policy (\ref{policy}), in order to be consistent with the  previous MFC algorithm setup. Then we consider to solve $\frac{d}{dt} Q_t = \mathcal{T}(Q_t,\tilde{\mu}_{Q_{t}})$,
and it gives the eventual solution satisfying $\mathcal{T}(Q_{\infty}, \tilde{\mu}_{Q_{\infty}})=0$.  
The pair $(Q_{\infty},\tilde{\mu}_{Q_{\infty}})$ solves (\ref{eqn:soln}) for MFC Bellman equation, and thus we call $\rho^{Q}\ll\rho^{\mu}$ the MFC regime. 
%where $Q_{\infty}$ satisfies the {\color{blue} MFC Bellman equation}.

\medskip

To better illustrate the above heuristics based on averaging, we consider a simple example here to  illustrate how a two-timescale algorithm can produce  different stationary solutions under different limiting ratios of learning rates. 
Let $Q, \mu$ both be scalar numbers, and for the updates (\ref{eqn:alg}), we consider
\begin{align}\label{eqn:PT}
    \mathcal{P}(Q,\mu) &= (Q-1)(\mu-Q),\\
    \mathcal{T}(Q,\mu)&=-(\mu-\frac12)(\mu -Q+1).
\end{align}
The Jacobian matrix is thus given by 
\begin{equation}
    J(Q, \mu) = 
    \begin{bmatrix}
           \frac{\partial \mathcal{P}}{\partial \mu} &   \frac{\partial \mathcal{P}}{\partial Q}\\
           \frac{\partial \mathcal{T}}{\partial \mu} & 
           \frac{\partial \mathcal{T}}{\partial Q}
    \end{bmatrix} =
    \left[
       \begin{array}{cc}
           Q-1 &   \mu -2Q+1\\
           -2\mu +Q -\frac12 & 
           \mu -\frac12
       \end{array}
        \right].
    \end{equation}
When $\rho^{\mu}\ll \rho^{Q}$, we  consider the approximate continuous-time ODEs
\begin{equation}\label{eq:Qeps}
    \begin{aligned}
    \frac{d}{d t} \mu_t &= \mathcal{P}(Q_t, \mu_t),\\
    \frac{d}{d t} Q_t  &= \frac{1}{\eps } \mathcal{T}(Q_t, \mu_t).
    \end{aligned}
\end{equation}
As $\eps \to 0$, $\mu_t$ can be assumed to be fixed. We first solves $\mathcal{T}(Q_{\mu}, \mu)=0$ and  obtain $Q_{\mu}= \mu + 1$. With such $Q_\mu$ plugged in to solve $\mathcal{P}(Q_{\mu_\infty},\mu_\infty)=0$, we obtain the fixed point solution to be  $(Q_{\mu_\infty},\mu_\infty) = (1, 0)$.

On the other hand, when $\rho^Q\ll \rho^{\mu}$, we have 
\begin{equation}\label{eq:mueps}
    \begin{aligned}
    \frac{d}{d t} \mu_t &= \frac{1}{\eps} \mathcal{P}(Q_t, \mu_t),\\
    \frac{d}{d t} Q_t  &= \mathcal{T}(Q_t, \mu_t).
    \end{aligned}
\end{equation}
As $\eps \to 0$, $Q_t$ can be assumed to be fixed. We thus first solve $\mathcal{P}(Q, \mu_Q)=0$, which gives  $\mu_Q = Q$. Then with such $\mu_Q$ plugged in to solve $\mathcal{T}(Q_{\infty},\mu_{Q_{\infty}})=0$, we obtain that $(Q_{\infty},\mu_{Q_{\infty}}) = (\frac12, \frac12)$, which is different from the one of \eqref{eq:Qeps}.
%%%%
It is easy to verify that the above two fixed points are both stable, while the system in fact also has a third fixed point $(Q,\mu) = (1, \frac{1}{2})$ which is unstable.
Thus the different ratio of the dynamics serves as a selection mechanism of different stable equilibria.

% How does the solutions of two timescale updates \eqref{eq:muQ} become different?
% Based on the chain rule
% \begin{equation}
%     \frac{d}{dt} Q =  \nabla_{\mu} Q \cdot  \frac{d}{dt} \mu,
% \end{equation}
% it can be observed that,
% for system~\eqref{eq:Qeps}, $Q$ solves
% \begin{equation}
% \mathcal{T}(Q, \mu)  -  \nabla_{\mu} Q \cdot \eps \mathcal{P}(Q, \mu) =0,
% \end{equation}
% and
% for system~\eqref{eq:mueps}, $Q$ solves
% \begin{equation}
% \eps \mathcal{T}(Q, \mu)  -  \nabla_{\mu} Q \cdot  \mathcal{P}(Q, \mu) =0.
% \end{equation}
% %
% $Q$ satisfies different equations with different rations of learning steps.%: $\eps$ and $\frac{1}{\eps}$.
% When $\eps \ll 1,$ in \eqref{eq:Qeps},
% $Q$ runs much faster than $\mu$, and satisfies  $\mathcal{T}(Q, \mu) =0$ with almost fixed $\mu$.
% While in \eqref{eq:mueps},
% $Q$ runs much slower than $\mu$, and satisfies  $\mathcal{T}(Q, \mu_Q) =0$ with $\mu_Q$ satisfying $\mathcal{P}(Q, \mu_Q) =0$ first.

We present the numerical simulation of the two-timescale algorithm with this toy construction~\eqref{eqn:PT}.
Figure~\ref{fig:toy-example} shows the trajectories of $Q,\mu$ respectively under various initializations and different learning rates.
By setting $\rho^{\mu}=0.001, \rho^Q=1$, $\mu_k$ runs slower than $Q_k$, which corresponds to the scenario~\eqref{eq:Qeps}, 
the algorithm converges to the solution $Q=1,\mu=0$.
On the other hand, by setting $\rho^{\mu}=1, \rho^Q=0.001$ so that $Q_k$ runs slower than $\mu_k$, which corresponds to the scenario~\eqref{eq:mueps}, the algorithm converges to the solution $Q=\frac{1}{2},\mu=\frac{1}{2}$.
We present the results with different initializations $(Q_0,\mu_0)$, and
simulation results show that the two-timescale algorithm is insensitive to initializations and converges to the fixed points determined by the step size ratios.

\begin{figure}[!ht]
\centering 
\begin{subfigure}[t]{0.32\textwidth}
		\centering
		\includegraphics[width=0.99\textwidth]{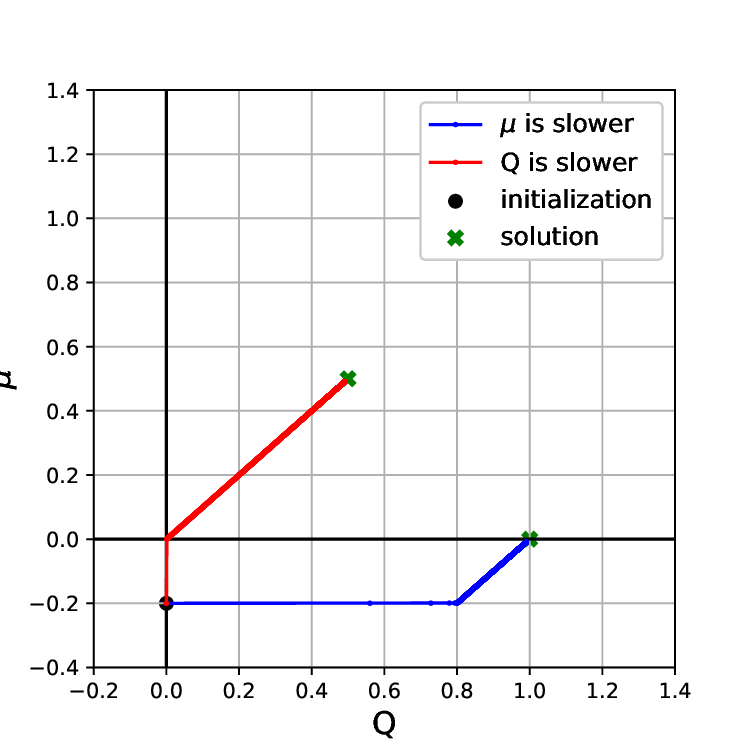}
		\caption{$Q_0=0, \mu_0=-0.2$}
		\label{fig:in1}
	\end{subfigure}
	\begin{subfigure}[t]{0.32\textwidth}
		\centering
		\includegraphics[width=0.99\textwidth]{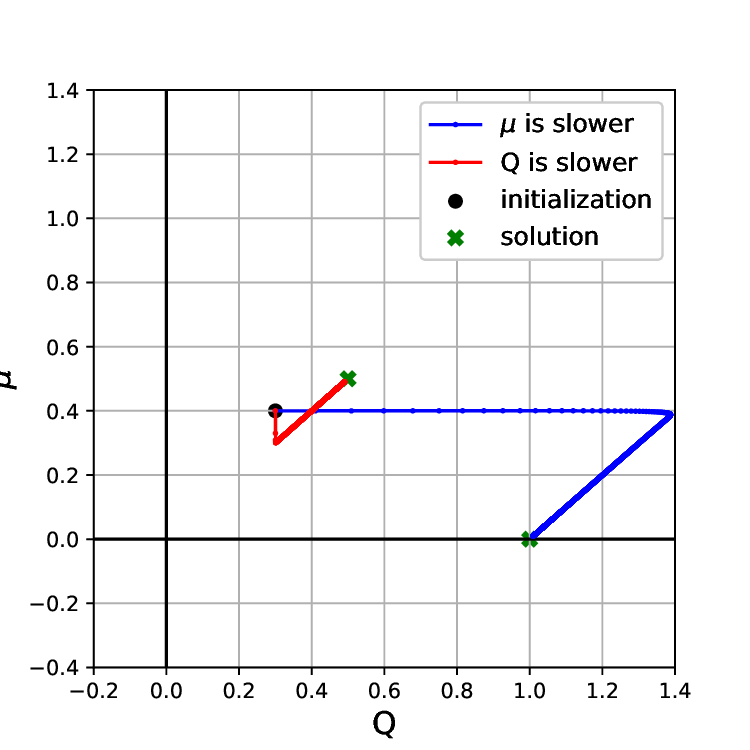}
		\caption{$Q_0=0.3, \mu_0=0.4$}
		\label{fig:in2}
	\end{subfigure}	
        \begin{subfigure}[t]{0.32\textwidth}
		\centering
		\includegraphics[width=0.99\textwidth]{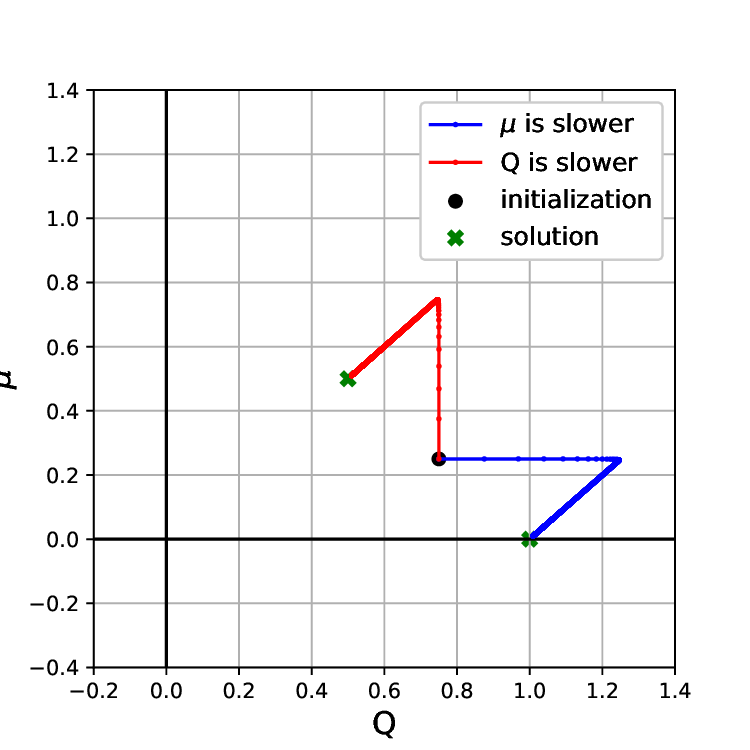}
		\caption{$Q_0=0.75, \mu_0=0.25$}
		\label{fig:in3}
	\end{subfigure}	
 \caption{Visualization of $(Q,\mu)$ trajectories in the two-timescale algorithm: The learning rates are set as $\rho^Q = 1, \rho^{\mu} = 0.001$ so that $\mu$ runs slower than $Q$, and the learning rates are $ \rho^Q = 0.001, \rho^{\mu} = 1$ so that $Q$ runs slower than $\mu$.}	
	\label{fig:toy-example}
\end{figure}

\FloatBarrier

\section{Unified convergence analysis}
\label{sec:conv}
% Recall the two timescale Q-learning algorithm
% \begin{equation}\label{eqn:alg}
%     \begin{aligned}
%     \mu_{k+1} &= \mu_k+\rho_k^{\mu} \mathcal{P}(Q_k, \mu_k),\\
%     Q_{k+1} &= Q_k + \rho_k^{Q} \mathcal{T}(Q_k, \mu_k),
%     \end{aligned}
% \end{equation}
% with 
% \begin{equation}
%     \begin{aligned}
%        \mathcal{P}(Q, \mu)(x) &= (\mu P^{Q,\mu})(x)-\mu(x),~~\text{for}~x\in\mathcal{X},\\
%        (\mu P^{Q,\mu})(x) & = \sum_{x_0} \mu(x_0) P^{Q,\mu}(x_0,x), 
%        \quad P^{Q,\mu}(x,x') = p(x'\mid x,\argmin_a Q(x,a),\mu),\\
%     %
%     %
%    \mathcal{T}(Q, \mu)(x,a)&= hf(x,a,\mu)+e^{-\gamma h} \sum_{x'}   p(x'\mid x, a,\mu)\min_{a'} Q(x',a')-Q(x,a),~~\text{for}~(x,a)\in \mathcal{X}\times \mathcal{A},
%     \end{aligned}
% \end{equation}
In this section, we provide a unified convergence analysis of the two-timescale Q-learning algorithm (\ref{eqn:alg}) for fixed learning rates $\rho^Q, \rho^\mu>0$ covering all ratios $\rho^Q/\rho^\mu \in (0,\infty)$.

Our approach for establish unified convergence relies on the following Lyapunov function, inspired by the idea of \cite{ZhouLu2023} for the analysis of single-timescale actor-critic method for the linear quadratic regulator problem. The Lyapunov function that we consider is
\begin{equation}\label{lyapunov}
    \mathcal{L}_k : = \mathcal{L}(\mu_k, Q_{h, k}) = W\min_{Q_h^*\in F}\|Q_{h, k}- Q_h^*\|_\infty+\|\mu_k-\tilde \mu_{k}\|_{\text{TV}},
\end{equation}
where $Q_{h,k}, \mu_k$ are numerical updates from the two-timescale Q-learning algorithm (\ref{eqn:alg}), $F$ is a set of fixed points solving the Bellman equation
\begin{equation}\label{eqn:conv_bellman}
    Q_h^*(x,a)=hf(x,a,\tilde{\mu}^*)+e^{-\gamma h} \sum_{x'}   p(x'\mid x, a,\tilde{\mu}^*)\min_{a'} Q_h^*(x',a'),
\end{equation}
coupled with 
\begin{equation}\label{eqn:conv_tildemu}
    \tilde{\mu}^*(x) = \tilde{\mu}^* P^{Q_h^*, \tilde{\mu}^*}(x).
\end{equation}
Here, we abuse the notations $Q_h^*, \tilde{\mu}^*$ to represent fixed points, and they are independent from previous sections. Moreover, $\tilde\mu_k$ is an intermediate equilibrium distribution corresponding to each $Q_{h,k}$ so that 
\begin{equation}\label{eqn:stationary_muk}
    \tilde \mu_k= \tilde \mu_{k} P^{Q_{h,k},\tilde \mu_{k}}.
\end{equation}
The weight parameter $W >0$ is to balance two discrepancies in (\ref{lyapunov}), it will actually be chosen depending on the ratio  $\rho^Q/\rho^\mu$.  
One cannot consider the Q-function update and the distribution update separately, since the operators $\mathcal{P}(Q,\mu)$ and $\mathcal{T}(Q,\mu)$ are highly coupled. 
Therefore, we devise such a Lyapunov function (\ref{lyapunov}) in order to capture the global convergence of $Q$ functions and local convergence of $\mu$ at the same time.  We will establish contraction
of the Lyapunov function $\mathcal{L}_k$ following the algorithm (\ref{eqn:alg}).

%$Q_h^*$, 
%which satisfies the Bellman equation
%\begin{equation}\label{dec04}
%    Q_{h}^*(x,a)= h f(x,a,\mu )
%		+ e^{-\gamma h} \sum_{x'\in\mathcal{X}}p(x'\mid x,a,\mu)\min_{a'} Q_{h}^*(x',a'),
%\end{equation}
%where {\color{red} $\mu =\mu^*$ for MFC problems , ?$\mu =\cL[x]$   }. 

\subsection{Assumptions and main result}

For our main result, we need following technical assumptions, which are standard and appear often in the analysis of convergence of Markov processes, see e.g., books like  \cite{meyn2012markov}. 
\begin{assumption}\label{assump:p}
There exist  $0<L_p<1,~ L_Q>0$ such that for all $x\in\mathcal{X}, \mu_1,\mu_2\in\mathcal{P}(\mathcal{X}), Q_1,Q_2$, we have the Lipschitz continuity 
    $$\sum_{x'\in\mathcal{X}}\big|P^{Q_1,\mu_1}(x,x') - P^{Q_2,\mu_2}(x,x')\big|\leq L_p \|\mu_1-\mu_2\|_{\text{TV}} +L_Q \|Q_1-Q_2\|_{\infty}.$$
\end{assumption}

% We aim at showing the convergence of two timescale algorithm to the fixed point $(Q^*,\mu^*)$ that solves the Bellman equation: 
% \begin{equation}\label{dec04}
%     Q_{h}^*(x,a)= h f(x,a,\mu^* )
% 		+ e^{-\gamma h} \sum_{x'\in\mathcal{X}}p(x'\mid x,a,\mu^*)\min_{a'} Q_{h}^*(x',a').
% \end{equation}
% %
% %
% %
% It is well-known \cite{risken1996fokker} that for a Markov process $x_t$ following the transition probability $P^{Q,\mu}$, we can derive the corresponding Fokker-Planck equation\begin{equation}\label{eqn:FP}
%         \frac{\partial }{\partial t }\mu_t = -\nabla (b(x_t,\alpha_t) \cdot \mu_t) +\frac12 \sigma^2 \Delta \mu_t
%     \end{equation}
%  with $b(x_t,\alpha_t)$ being the time derivative of the expectation $\d_t \langle x(t;x_0)\rangle$ and $\sigma^2$ being the time derivative of the variance of $x_t$.

% \begin{definition}
%     We define the Dirichlet energy for a transition probability $P$ as
%     \begin{equation}
%         \mathcal{E}_P(f,g) = \langle f(I-P), g\rangle.
%     \end{equation}
% \end{definition}

\begin{assumption} (Uniform Doeblin’s condition)\label{assump:doeblin}
% We assume that the Lebesgue measure of the state space $\mathcal{X}$ is bounded, denoted by
% \begin{equation}
%     N_{\mathcal{X}} := |\mathcal{X}|<+\infty.
% \end{equation}
We assume that for any bounded $Q$,  the transition probability $
        P^{Q,\mu}(x,x') = p(x' \mid x,\argmin_a Q(x,a),\mu)$ has an equilibrium probability measure $\pi$ that solves $\pi = \pi P^{Q, \pi}$.  There exist a constant $$\beta\in\Big(\frac{1+L_p}{2},1\Big)$$ and probability measure $\nu$ such that
\begin{equation}\label{eq:doeblin} 
    P^{Q,\pi}(x,\cdot)\geq \beta \nu(\cdot),
\end{equation}
for all $x\in\mathcal{X}$.
\end{assumption}

The above assumptions guarantee that the asymptotic distribution is unique for a given $Q$ function (and hence fixed policy):
\begin{proposition}\label{prop:brouwer}
With Assumptions \ref{assump:p} and \ref{assump:doeblin},  for a fixed $Q_{h,k}$, there exists a unique equilibrium distribution solving \eqref{eqn:stationary_muk}. In addition with Assumptions \ref{assump:f}, there exists a pair of fixed points $(Q_h^*, \tilde\mu^*)$ solving \eqref{eqn:conv_bellman} and \eqref{eqn:conv_tildemu}.
\end{proposition}
% In addition with Assumptions \ref{assump:f},  with sufficiently small $h$ and large $\gamma$ such that $\gamma h \gg 1$, there exists a unique fixed point $(Q_h^*, \tilde\mu^*)$ solving \eqref{eqn:conv_bellman} and \eqref{eqn:conv_tildemu}.

% Furthermore, by assuming that for any $x\in\mathcal{X}$, there exists a unique action $a\in\mathcal{A}$ such that the optimal $Q$-function achieves its minimum, then we can find $Q_h^*$ satisfies the Bellman equation \eqref{eqn:conv_bellman}, and correspondingly the equilibirum distribution $\tilde{\mu}^*$ satisfying \eqref{eqn:conv_tildemu} exists and is unique.
    % We assume that for the fixed point $Q_h^*$ solving \eqref{eqn:conv_bellman}, there exists a unique $\tilde \mu^*$ such that
    % \begin{equation}
    % \tilde\mu^*  = \tilde\mu^* P^{Q_h^*,\tilde\mu^*}.

We defer the proof of the Proposition after Lemma \ref{lem:mu_and_Q}.

\begin{remark}
Except for a certain parameter regime where both $h$ and $e^{-\gamma h}$ are sufficiently small, the target fixed points $(Q_h^*, \tilde\mu^*)$ in general are not unique. This is why in the Lyapunov function (\ref{lyapunov}), we consider the minimum distance over all possible $Q_h^*$, since we do not know which pair of fixed points that the algorithm with fixed constant learning rates can converge to. For completeness, we leave the analysis of uniqueness parameter regime to Appendix \ref{sec:uniqueness}.
\end{remark}

Equipped with assumptions above, we are ready to state the main result: the unified convergence of (\ref{eqn:alg}).
\begin{thm}\label{thm:main}
    With Assumptions \ref{assump:f}, \ref{assump:p}, and \ref{assump:doeblin}, we require learning rates  to satisfy that
    \begin{equation}
        0<\rho^{\mu}<\frac{1}{2\beta-1-L_p},\quad 0<\rho^Q<\frac{1}{1-e^{-\gamma h}}.
    \end{equation}
   With these bounds, and taking $\Lambda_{\mu}= 1- \rho^{\mu}\big(2\beta-1-L_p\big)$, the Lyapunov function (\ref{lyapunov}) under the two-timescale Q-learning algorithm (\ref{eqn:alg}) contracts as 
    \begin{equation}\label{eq:mar27}
         \mathcal{L}_k \leq (1-c)^k \mathcal{L}_0+ \frac{2h\Lambda_{\mu}L_Q\rho^Q \|f\|_{\infty}}{c(1-e^{-\gamma h})(2\beta-1-L_p)}, 
    \end{equation}
    where  $c = \min\{c_1, c_2\}$ with 
    \begin{equation}
\begin{aligned}
    &c_1= \rho^Q\Big(1-e^{-\gamma h}-  \Big(L_f + \frac{L_p\|f\|_{\infty}}{e^{\gamma h}-1}  \Big)\frac{hL_Q}{2\beta-1-L_p} \Big)-\frac{2\Lambda_\mu L_Q}{W(2\beta-1-L_p) },\\
    &c_2= 1-\Lambda_{\mu} -W\rho^Q h \big(L_f + \frac{L_p\|f\|_{\infty} }{e^{\gamma h}-1} \big).
    \end{aligned}
\end{equation}
As $k\to \infty$, we have that $ \mathcal{L}_\infty = O(\rho^Q)$.
\end{thm}

Our convergence result significantly extends that of \cite{angiuli2023convergence}, which only considered extreme regimes where $\lim_{k\to \infty}\rho_k^Q/\rho_k^\mu = 0$ and $\lim_{k\to \infty}\rho_k^Q/\rho_k^\mu=\infty$, and applied convergence results from \cite{borkar1997stochastic} directly with no quantitative convergence rates. 
Here we choose $\rho_k^\mu, \rho_k^Q$ to be fixed constants $\rho^{\mu}, \rho^Q$ rather than of the Robbins-Monro type  as in \cite{borkar1997stochastic} for simplicity. We believe our result can be extended to Robbins-Monro type learning rates as well with some modifications. 
\begin{remark}
Our quantitative convergence result sheds insight on how the contraction rate $1-c$ depends on learning rates $\rho^Q, \rho^\mu$ precisely. The choices of $c$ in (\ref{eq:mar27}) illustrate the dichotomy convergence behaviors of the two-timescale Q-learning algorithm (\ref{eqn:alg}) when $\rho^Q\gg \rho^{\mu}$ or $\rho^Q\ll \rho^{\mu}$, thus get connected to MFG and MFC regimes.
\begin{itemize}
    \item In the MFG regime where $\rho^Q\gg \rho^{\mu}$, recall $\Lambda_{\mu} =1- \rho^{\mu}\big(2\beta-1-L_p\big)$, we need to ensure 
    \begin{equation}
        c_2 = \rho^{\mu}\big(2\beta-1-L_p\big) - W\rho^Q h \big(L_f + \frac{L_p\|f\|_{\infty} }{e^{\gamma h}-1} \big) >0,
    \end{equation}
    which implies that
    \begin{equation}
        W< \frac{\rho^{\mu}\big(2\beta-1-L_p\big)}{\rho^Q h \big(L_f + \frac{L_p\|f\|_{\infty} }{e^{\gamma h}-1} \big) } \ll 1.
    \end{equation}
    It means that  the convergence of $\mu$ dominates the convergence of \eqref{eqn:alg}, which is aligned with the fact that we have fast convergence for $Q$ and slow convergence for $\mu$.
    \item In the MFC regime where $\rho^Q\ll \rho^{\mu}$, we need to ensure $c_1>0$ so that 
    \begin{equation}
    \begin{aligned}
        W&>\frac{2\Lambda_\mu L_Q}{\rho^Q(2\beta-1-L_p) \Big(1-e^{-\gamma h}-  \Big(L_f + \frac{L_p\|f\|_{\infty}}{e^{\gamma h}-1}  \Big)\frac{hL_Q}{2\beta-1-L_p} \Big)}.
\end{aligned}
    \end{equation}
    It means that for sufficiently small $\rho^Q$, we need to put a larger weight on $Q$ convergence so that $Q$ dominates the whole process. It is aligned with our observation that in the MFC regime, we have fast convergence for $\mu$ and slow convergence for $Q$.
\end{itemize}
\end{remark}
\subsection{Auxiliary results}

We first present some auxiliary results before proving Theorem~\ref{thm:main}. The following theorem investigates the case when $Q_{h,k}\equiv Q$ is fixed.

\begin{proposition}\label{prop:conv_mu}
Let $\mu_k$ update as in \eqref{eqn:alg} with $Q_{h,k}\equiv Q$. Suppose $\tilde\mu$ is the equilibrium probability measure such that $\tilde \mu = \tilde\mu P^{Q,\tilde \mu}$, the transition probability satisfies Assumption~\ref{assump:p},  then given a fixed step size $\rho^\mu$ that satisfies
\begin{equation}
    0<\rho^{\mu}< \frac{1}{2\beta-1-L_p}
\end{equation}
and $\beta$ provided in Assumption \ref{assump:doeblin}, we can find a contraction rate $\Lambda_\mu =  1- \rho^{\mu}\big(2\beta-1-L_p\big)$ such that for all $k\geq 0$,
\begin{equation}\label{eqn:conv_mu}
    \|\mu_{k+1}-\tilde \mu\|_{\text{TV}} \leq \Lambda_\mu \|\mu_k-\tilde \mu\|_{\text{TV}}.
\end{equation}
\end{proposition}

\begin{proof}
The $\mu_k$ update step in \eqref{eqn:alg} gives that
\begin{equation}
    \begin{aligned}
        \mu_{k+1} -\tilde{\mu}&= \mu_{k} -\tilde{\mu}+\rho^{\mu}\big(\mu_k P^{Q,\mu_k}-\mu_k\big)\\
           &  =     \mu_{k} -\tilde{\mu}+\rho^{\mu}\big((\mu_k -\tilde \mu )P^{Q,\tilde \mu}+\mu_k(P^{Q,\mu_k}-P^{Q,\tilde\mu})-(\mu_k-\tilde \mu) \big)\\
           & = (1-\rho^\mu)(\mu_{k} -\tilde{\mu}) + \rho^{\mu}\big((\mu_k -\tilde \mu )P^{Q,\tilde \mu}+\mu_k(P^{Q,\mu_k}-P^{Q,\tilde\mu})\big).
    \end{aligned}
\end{equation}
Using the triangle inequality, we have that
\begin{equation}\label{jan293}
    \begin{aligned}
       \|\mu_{k+1} -\tilde{\mu}\|_{\text{TV}}&\leq  (1-\rho^\mu)\|\mu_{k} -\tilde{\mu}\|_{\text{TV}} + \rho^{\mu}\|(\mu_k -\tilde \mu )P^{Q,\tilde \mu}\|_{\text{TV}}\\
       &\quad + \rho^{\mu} \|\mu_k(P^{Q,\mu_k}-P^{Q,\tilde\mu})\|_{\text{TV}} = (1-\rho^\mu)\|\mu_{k} -\tilde{\mu}\|_{\text{TV}}+I+II.
    \end{aligned}
\end{equation}
To treat $I$, we decompose $\mu_k-\tilde\mu$ as
\begin{equation}
    \mu_k-\tilde\mu = (\mu_k-\tilde\mu)_+ -(\mu_k-\tilde\mu)_- =: S_+ - S_-.
\end{equation}
Note that for any $A\subseteq \mathcal{X}$, by (\ref{eq:doeblin}), 
\begin{equation}\label{jan291}
    \sum_{x\in A}\sum_{x_0\in\mathcal{X}}S_{\pm}(x_0)P^{Q,\tilde \mu}(x_0,x)\geq \beta \sum_{x\in A}\sum_{x_0\in\mathcal{X}}S_{\pm}(x_0) \nu(x),
\end{equation}
and 
\begin{equation}\label{jan292}
    \sum_{x\in A}\sum_{x_0\in\mathcal{X}}S_{+}(x_0) \nu(x)= \sum_{x\in A}\sum_{x_0\in\mathcal{X}}S_{-}(x_0) \nu(x).
\end{equation}
Therefore, we use (\ref{jan292}) and apply the triangle inequality to get
\begin{equation}
    \begin{aligned}
         &\|(\mu_k-\tilde\mu)P^{Q,\tilde \mu} \|_{\text{TV}} =
         \sup_{A\subseteq \mathcal{X}}\Big| \sum_{x\in A} S_+P^{Q,\tilde \mu}(x) - S_-P^{Q,\tilde \mu}(x)\Big|\\
         &\leq \sup_{A\subseteq \mathcal{X}}\Big| \sum_{x\in A} S_+P^{Q,\tilde \mu}(x) - \beta\sum_{x\in A}\sum_{x_0\in\mathcal{X}}S_{+}(x_0) \nu(x)\Big|+\sup_{A\subseteq \mathcal{X}}\Big| \sum_{x\in A} S_-P^{Q,\tilde \mu}(x) - \beta\sum_{x\in A}\sum_{x_0\in\mathcal{X}}S_{-}(x_0) \nu(x)\Big|\\
&= \sum_{x\in \mathcal{X}}\Big( S_+P^{Q,\tilde \mu}(x) - \beta\sum_{x_0\in\mathcal{X}}S_{+}(x_0) \nu(x)\Big) + \sum_{x\in \mathcal{X}}\Big( S_-P^{Q,\tilde \mu}(x) - \beta\sum_{x_0\in\mathcal{X}}S_{-}(x_0) \nu(x)\Big)\\
&= \sum_{x\in \mathcal{X}}\sum_{x_0\in\mathcal{X}}(S_+ + S_-)(x_0)P^{Q,\tilde \mu}(x_0,x)-\beta \sum_{x\in \mathcal{X}}\sum_{x_0\in\mathcal{X}}(S_+ + S_-)(x_0)\nu(x)\\
&= (1-\beta)\sum_{x_0\in\mathcal{X}} |\mu_k(x_0)-\tilde\mu(x_0)| = (1-\beta) \|\mu_k-\tilde\mu\|_1 = 2(1-\beta)\|\mu_k-\tilde\mu \|_{\text{TV}},
\end{aligned}
\end{equation}
where the second equality above uses (\ref{jan291}).

For $II$, we use Assumption \ref{assump:p} to get
\begin{equation}
    \begin{aligned}
     \|\mu_k(P^{Q,\mu_k}-P^{Q,\tilde\mu})\|_{\text{TV}} &= \sup_{A\subseteq \mathcal{X}}\Big| \sum_{x\in A}\sum_{x_0\in\mathcal{X}} \mu_k(x_0)\big(P^{Q,\mu_k}(x_0,x)-P^{Q,\tilde\mu}(x_0,x)\big)\Big|\\
     &\leq \sum_{x\in \mathcal{X}}\sum_{x_0\in\mathcal{X}} \mu_k(x_0)\Big|P^{Q,\mu_k}(x_0,x)-P^{Q,\tilde\mu}(x_0,x)\Big|\\
     &\leq L_p \|\mu_k-\tilde{\mu}\|_{\text{TV}}\sum_{x_0\in\mathcal{X}} \mu_k(x_0) =  L_p \|\mu_k-\tilde{\mu}\|_{\text{TV}}.
    \end{aligned}
\end{equation}

Combining all parts together, we have
\begin{equation*}
     \|\mu_{k+1} -\tilde{\mu}\|_{\text{TV}} \leq \Big(1- \rho^{\mu}\big(2\beta-1-L_p\big)\Big)\|\mu_k-\tilde{\mu}\|_{\text{TV}}.  
\end{equation*}
% note that
% \begin{equation}
% \begin{aligned}
%     \|\mu_kP^{Q,\tilde \mu} - \beta\nu\|_{\text{TV}} &= \sup_{A\subseteq \mathcal{X}} \sum_{x\in A}(\mu_kP^{Q,\tilde \mu}(x) - \beta\nu(x)) = \sum_{x\in \mathcal{X}}(\mu_kP^{Q,\tilde \mu}(x) - \beta\nu(x))\\
%     &=\sum_{x\in \mathcal{X}}\sum_{x_0\in \mathcal{X}}\mu_k(x_0)P^{Q,\tilde \mu}(x_0, x) - \beta = 1-\beta.
%     \end{aligned}
% \end{equation}
% Similarly, $\|\tilde\mu P^{Q,\tilde\mu} - \beta\nu\|_{\text{TV}}=1-\beta$. Moreover, recall the fact that for any two probability measures $\nu_1, \nu_2$, by writing $\nu_1-\nu_2 = (\nu_1-\nu_2)_+- (\nu_2-\nu_1)_+$ to have disjoint support, we have that
% \begin{equation}
%     \|\nu_1-\nu_2\|_{\text{TV}} = 2
% \end{equation}
\end{proof}

Now let $Q_{h,k}$ update as in the algorithm, we have the following iteration bound for $Q_{h,k}$.
\begin{proposition}\label{prop:conv_Q}
With Assumptions \ref{assump:f} and \ref{assump:p}, we can find a contraction rate $\Lambda_Q=1-\rho^Q(1-e^{-\gamma h}) \in(0,1)$ such that the maximal difference between $Q_{h,k}$ and a fixed point $(Q_h^*, \tilde{\mu}^{\ast})$ in \eqref{eqn:conv_bellman}--\eqref{eqn:conv_tildemu} iterates as
  \begin{equation}\label{jan294}
  \begin{aligned}
\|{Q}_{h,k+1} -  Q_{h}^*\|_{\infty} & \leq (1-\rho^Q(1-e^{-\gamma h}))\|{Q}_{h,k} -  Q_{h}^*\|_{\infty} \\
& \qquad +\rho^Q h  \|\mu_k-\tilde{\mu}^*\|_{\text{TV}} \Big(L_f + \frac{L_p}{e^{\gamma h}-1} \|f\|_{\infty} \Big).
  \end{aligned}
  \end{equation}
\end{proposition}

\begin{proof}
    We rewrite the absolute value difference 
    \begin{equation}
    \Delta_{k}(x,a) := |{Q}_{h,k}(x,a) -  Q_{h}^*(x,a)|
    \end{equation}
for shortness. The two-timescale Q-learning \eqref{eqn:alg} gives that
\begin{equation}
\begin{aligned}
    Q_{h,k+1}&(x,a)-Q_{h}^*(x,a)= Q_{h,k}(x,a)-Q_{h}^*(x,a) + \rho^Q\mathcal{T}(Q_{h,k},\mu_k)\\
    &= (1-\rho^Q)(Q_{h,k}(x,a)-Q_{h}^*(x,a)) + \rho^Q\Big(h f(x,a,\mu_k)-h f(x,a,\tilde{\mu}^*)\\
    &\quad  +e^{-\gamma h}\sum_{x'\in\mathcal{X}}p(x'\mid x,a,\mu_k)\inf_{a'} Q_{h,k}(x',a')-e^{-\gamma h}\sum_{x'\in\mathcal{X}}p(x'\mid x,a,\tilde{\mu}^*)\inf_{a'} Q_h^*(x',a')\Big).
\end{aligned}
\end{equation}
By the triangle inequality, we get that
% if we take the supremum on the right side first and left side later, then we will get
\begin{equation}
\begin{aligned}
    \Delta_{k+1}(x,a)& \leq (1-\rho^Q)\Delta_k(x,a) + \rho^Q h |f(x,a,\mu_k)-f(x,a,\tilde{\mu}^*)|  \\
    &~~~ +\rho^Q e^{-\gamma h}\Big|\sum_{x'\in\mathcal{X}}p(x'\mid x,a,\mu_k)\inf_{a'} Q_{h,k}(x',a')-\sum_{x'\in\mathcal{X}}p(x'\mid x,a,\tilde{\mu}^*)\inf_{a'} Q_h^*(x',a')\Big|\\
    & \leq (1-\rho^Q)\Delta_k(x,a) + I+II+III,
\end{aligned}
\end{equation}
where 
\begin{align*}
    I  &=  \rho^Q h |f(x,a,\mu_k)-f(x,a,\tilde{\mu}^*)|, \\
    II &= \rho^Q e^{-\gamma h}\Big|\sum_{x'\in\mathcal{X}}p(x'\mid x,a,\mu_k)\inf_{a'} Q_{h,k}(x',a')-\sum_{x'\in\mathcal{X}}p(x'\mid x,a,\tilde{\mu}^*)\inf_{a'} Q_{h,k}(x',a')\Big|,\\
    III & = \rho^Q e^{-\gamma h}\Big|\sum_{x'\in\mathcal{X}}p(x'\mid x,a,\tilde{\mu}^*)\inf_{a'} Q_{h,k}(x',a')-\sum_{x'\in\mathcal{X}}p(x'\mid x,a,\tilde{\mu}^*)\inf_{a'} Q_h^*(x',a')\Big|.
\end{align*}
For $I$, we use the Assumption \ref{assump:f} to get
\begin{equation}\label{dec060}
    I \leq \rho^Q h L_f \|\mu_k-\tilde{\mu}^*\|_{\text{TV}}.
\end{equation}
For $II$, we use the Assumption \ref{assump:p} to get
\begin{equation}\label{dec061}
   II \leq \rho^Q e^{-\gamma h} \|Q_{h,k}\|_{\infty} L_p\|\mu_k-\tilde{\mu}^*\|_{\text{TV}} \leq \rho^Q \frac{ L_p h}{e^{\gamma h}-1} \|f\|_{\infty}  \|\mu_k-\tilde{\mu}^*\|_{\text{TV}},
\end{equation}
since by the definition of discrete time Q-function, 
\begin{equation}\label{eq:boundQ}
    \|Q_{h,k}\|_{\infty}\leq \frac{h}{1-e^{-\gamma h}}\|f\|_{\infty}.
\end{equation}
For $III$, we have
\begin{equation}\label{dec062}
    III\leq \rho^Q e^{-\gamma h}\sup_{a'}\sup_{x'\in\mathcal{X}}\Delta_k(x',a').
\end{equation}
Now combining all bounds of $I, II, III$ together, we can take the supremum over $(x, a)\in \mathcal{X}\times \mathcal{A}$ on the right side first and left side later to obtain that
\begin{equation*}
    \|{Q}_{h,k+1} -  Q_{h}^*\|_{\infty}\leq (1-\rho^Q(1-e^{-\gamma h}))\|{Q}_{h,k} -  Q_{h}^*\|_{\infty}+\rho^Q h  \|\mu_k-\tilde{\mu}^*\|_{\text{TV}} \Big(L_f + \frac{L_p}{e^{\gamma h}-1} \|f\|_{\infty} \Big). 
\end{equation*}
% Let $\Lambda_Q = 1-\rho^Q(1-e^{-\gamma h})$, and $C_h = \rho^Q h \sqrt{N_{\mathcal{X}}} \|\mu_0-\tilde \mu\|_{L^2} \Big(L_f + \frac{1}{e^{\gamma h}-1} \|f\|_{L^{\infty}} L_p\Big)$, we eventually get that for any $k\geq 1$,
% \begin{equation}
%     \Delta_k \leq \Lambda_Q^k \Delta_0 + C_h \sum_{j=0}^{k-1} \Lambda_Q^{k-1-j} \Lambda_\mu^{j/2}  = \Lambda_Q^k \Delta_0+O(h).
% \end{equation}
\end{proof}

We do not know the relation between $\tilde\mu_k$ and $\tilde\mu^*$ a priori, since the equation $\mu = \mu P^{Q,\mu}$ is highly nonlinear. However, we can control 
the difference between two equilibrium distributions by the difference of their corresponding $Q$-functions, if one of $P^{Q,\mu}$ satisfies the uniform Doeblin's condition.

\begin{lemma}\label{lem:mu_and_Q}
   Given $Q_1, Q_2\in \Rm_+$, if $\mu_1, \mu_2\in\mathcal{P}(\mathcal{X})$ solves
    \begin{equation}
        \mu_1(I- P^{Q_1, \mu_1}) =0,\quad \mu_2(I- P^{Q_2, \mu_2}) =0
    \end{equation}
    respectively, then based on Assumptions \ref{assump:p} and \ref{assump:doeblin} of $P^{Q_2, \mu_2}$ (or $P^{Q_1, \mu_1}$),  we have the relation 
    \begin{equation}\label{eqn:uniqueness}
    \|\mu_1-\mu_2\|_{\text{TV}} \leq \frac{L_Q}{2\beta-1-L_p} \|Q_1-Q_2\|_{\infty}.
\end{equation}
\end{lemma}
\begin{proof}
The proof resembles the proof of Proposition \ref{prop:conv_mu}. Note that
\begin{equation}\label{jan2941}
    \|\mu_1-\mu_2\|_{\text{TV}} = \|\mu_1 P^{Q_1,\mu_1}-\mu_2P^{Q_2,\mu_2}\|_{\text{TV}} \leq\|(\mu_1-\mu_2)P^{Q_2,\mu_2}\|_{\text{TV}}+ \|\mu_1\big(P^{Q_1,\mu_1}-P^{Q_2, \mu_2}\big)\|_{\text{TV}}.
\end{equation}
The first term above can be treated in the same way as for part $I$ in (\ref{jan293}) to have the bound
\begin{equation}
    \|(\mu_1-\mu_2)P^{Q_2,\mu_2}\|_{\text{TV}} \leq 2(1-\beta) \|\mu_1-\mu_2\|_{\text{TV}}.
\end{equation}
The second term in (\ref{jan2941}) uses Assumption \ref{assump:p} so that
\begin{equation}
    \|\mu_1\big(P^{Q_1,\mu_1}-P^{Q_2, \mu_2}\big)\|_{\text{TV}} \leq L_p \|\mu_1-\mu_2\|_{\text{TV}} + L_Q\|Q_1-Q_2\|_{\infty}.
\end{equation}
Combining all terms together we have
\begin{equation}
    (2\beta-1-L_p)\|\mu_1-\mu_2\|_{\text{TV}} \leq L_Q\|Q_1-Q_2\|_{\infty}
\end{equation}
to obtain the relation.
\end{proof}
Based on Lemma \ref{lem:mu_and_Q}, we can control the difference between $\tilde{\mu}_k$ and $\tilde{\mu}^*$ to be
\begin{equation}\label{eq:mu_Q}
    \|\tilde \mu_k-\tilde{\mu}^*\|_{\text{TV}} \leq \frac{L_Q}{2\beta-1-L_p} \|Q_{h,k}- Q_h^*\|_\infty.
\end{equation}

\begin{proof}[Proof of Proposition \ref{prop:brouwer}] Consider the mapping $\mathcal{M}_Q:\mathcal{P}(\mathcal{X})\to \mathcal{P}(\mathcal{X})$ such that $\mathcal{M}_Q(\mu) = \mu P^{Q,\mu}$. For a given $Q$, $\mathcal{M}_Q$ is continuous since
\begin{equation}
\begin{aligned}
    \|\mathcal{M}_Q(\mu_1)-\mathcal{M}_Q(\mu_2)\|_{\text{TV}} &= \|\mu_1P^{Q,\mu_1}-\mu_2P^{Q,\mu_2}\|_{\text{TV}}\\
    &\leq \|(\mu_1-\mu_2)P^{Q,\mu_1}\|_{\text{TV}}+\|\mu_2(P^{Q,\mu_1}-P^{Q,\mu_2})\|_{\text{TV}}\\
    &\leq 2(1-\beta)\|\mu_1-\mu_2\|_{\text{TV}}+ L_p \|\mu_1-\mu_2\|_{\text{TV}}\\
    &= (2(1-\beta)+L_p)\|\mu_1-\mu_2\|_{\text{TV}}.
\end{aligned}
\end{equation}
Because $\mathcal{X}$ is finite, by Brouwer's fixed point theorem, there exists $\mu$ such that $\mathcal{M}_Q(\mu)=\mu$ for the given $Q$. This fixed point $\mu$ is unique due to \eqref{eqn:uniqueness}.
For equations \eqref{eqn:conv_bellman} and \eqref{eqn:conv_tildemu}, we define the Bellman operator
\begin{equation}\label{Bellman_operator}
    \mathcal{B}_{\mu} (Q)(x,a) := hf(x,a,\mu)+e^{-\gamma h} \sum_{x'}   p(x'\mid x, a,\mu)\min_{a'} Q(x',a').
\end{equation}
The mapping pair $(\mathcal{B}_{\mu} (Q), \mathcal{M}_Q(\mu)): \Rm_+\times \mathcal{P}(\mathcal{X})\to \Rm_+\times \mathcal{P}(\mathcal{X})$ is continuous since
\begin{equation}
\begin{aligned}
   \|\mathcal{B}_{\mu_1} (Q_1)-\mathcal{B}_{\mu_2} (Q_2)\|_\infty &\leq \|\mathcal{B}_{\mu_1} (Q_1)-\mathcal{B}_{\mu_2} (Q_1)\|_\infty+\|\mathcal{B}_{\mu_2} (Q_1)-\mathcal{B}_{\mu_2} (Q_2)\|_\infty  \\
   &\leq (hL_f+e^{-\gamma h}L_p\|Q\|_\infty)\|\mu_1-\mu_2\|_{\text{TV}}+e^{-\gamma h} \|Q_1-Q_2\|_\infty,
   \end{aligned}
\end{equation}
and 
\begin{equation}
\begin{aligned}
     \|\mathcal{M}_{Q_1}(\mu_1)-\mathcal{M}_{Q_2}(\mu_2)\|_{\text{TV}} &\leq \|\mathcal{M}_{Q_1}(\mu_1)-\mathcal{M}_{Q_2}(\mu_1)\|_{\text{TV}} +\|\mathcal{M}_{Q_2}(\mu_1)-\mathcal{M}_{Q_2}(\mu_2)\|_{\text{TV}} \\
     &\leq L_Q \|Q_1-Q_2\|_\infty+(2(1-\beta)+L_p)\|\mu_1-\mu_2\|_{\text{TV}}.
     \end{aligned}
\end{equation}
Thus by Brouwer's fixed point theorem, there exists a fixed point $(Q_h^*, \tilde\mu^*)$ solving \eqref{eqn:conv_bellman} and \eqref{eqn:conv_tildemu}.
% \JA{JA: replace assumption 5 by Brouwer fixed point theorem and Lemma 4.4: Existence is done by showing that for fixed $Q$,
% \begin{equation}
%   \mathcal{M}:  \mu \mapsto \mu P^{Q,\mu}
% \end{equation}
% is a continuous mapping}

% Because $\|\mathcal{B}_{\mu} (Q_1)-\mathcal{B}_{\mu} (Q_2)\|_\infty \leq e^{-\gamma h} \|Q_1-Q_2\|_\infty$, $Q\mapsto \mathcal{B}_{\mu} (Q)$ is a strict contraction, thus for a fixed $\mu$, there exsits a unique $Q$ satisfying $\mathcal{B}_{\mu} (Q) = Q$ by contraction mapping theorem. 
\end{proof}

\subsection{Proof of Theorem \ref{thm:main}}
\begin{proof}
The proof strategy is to obtain the iteration inequality in the form
\begin{equation*}
    \begin{aligned}
 \mathcal{L}_{k+1} -\mathcal{L}_k &=W\min_{Q_h^*\in F}\|Q_{h,k+1}- Q_h^*\|_\infty-W\min_{Q_h^*\in F}\|Q_{h,k}- Q_h^*\|_\infty \\
 & \qquad + \|\mu_{k+1}-\tilde \mu_{k+1}\|_{\text{TV}}-\|\mu_k-\tilde{\mu}_{k}\|_{\text{TV}}\\&\leq -c \mathcal{L}_k + e_k,
    \end{aligned}
\end{equation*}
with some $c\in(0,1)$ and bounded errors $e_k$. Then by iteration, we can get that for each $k\geq 1$, 
\begin{equation}
    \mathcal{L}_k \leq (1-c)^k \mathcal{L}_0+\sum_{j=0}^{k-1}(1-c)^{k-1-j} e_j.
\end{equation}
Throughout the remaining proof, we abuse the notation $Q_h^*$ to denote the minimizer of $\min_{Q_h^*\in F}\|Q_{h,k}- Q_h^*\|_\infty$ in $\mathcal{L}_k$, with such a $Q_h^*$, $\tilde\mu^*$ is also uniquely determined by Lemma~\ref{lem:mu_and_Q}.

\paragraph{Estimate on $\mu_k-\tilde\mu_k$}
The proof of Proposition~\ref{prop:conv_mu} implies that
\begin{equation}
    \begin{aligned}
 \|\mu_{k+1}-\tilde \mu_{k+1}\|_{\text{TV}}\leq \Lambda_{\mu} \|\mu_{k}-\tilde \mu_{k+1}\|_{\text{TV}}.
    \end{aligned}
\end{equation}
Therefore, in addition with the triangle inequality, we can write 
\begin{equation}\label{dec07}
    \begin{aligned}
         \|\mu_{k+1}&-\tilde \mu_{k+1}\|_{\text{TV}}-\|\mu_k-\tilde{\mu}_{k}\|_{\text{TV}}\leq\Lambda_\mu\|\mu_{k}-\tilde{\mu}_{k+1}\|_{\text{TV}}-\|\mu_k-\tilde{\mu}_{k}\|_{\text{TV}}\\& \leq \Lambda_\mu\|\tilde\mu_{k+1}-\tilde{\mu}_{k}\|_{\text{TV}} - (1-\Lambda_{\mu}) \|\mu_k-\tilde{\mu}_{k}\|_{\text{TV}}\\
         &\leq \Lambda_\mu \big( \|\tilde \mu_k-\tilde{\mu}^*\|_{\text{TV}}+ \|\tilde \mu_{k+1}-\tilde{\mu}^*\|_{\text{TV}}\big) - (1-\Lambda_{\mu}) \|\mu_k-\tilde{\mu}_{k}\|_{\text{TV}}\\
         &\leq \frac{\Lambda_\mu L_Q}{2\beta-1-L_p}\Big(\|Q_{h,k}- Q_h^*\|_\infty+\|Q_{h,k+1}- Q_h^*\|_\infty\Big)- (1-\Lambda_{\mu}) \|\mu_k-\tilde{\mu}_{k}\|_{\text{TV}}\\
         &\leq \frac{2\Lambda_\mu L_Q}{2\beta-1-L_p }\|Q_{h,k}- Q_h^*\|_\infty+\frac{\Lambda_\mu L_Q \rho^Q }{2\beta-1-L_p}\|\mathcal{T}(\cdot, \cdot)\|_{\infty}- (1-\Lambda_{\mu}) \|\mu_k-\tilde{\mu}_{k}\|_{\text{TV}},
    \end{aligned}
\end{equation}
where the last inequality is obtained by the Q iteration in \eqref{eqn:alg} and the uniform boundedness of the operator $\mathcal{T}$ over $Q,\mu$.

\paragraph{Estimate on $Q_{h,k}-Q_h^*$} 
Based on Proposition~\ref{prop:conv_Q}, we have that (recall that with a bit abuse of notation we denote $Q_h^*$ the minimizer of $\min_{Q_h^*\in F}\|Q_{h,k}- Q_h^*\|_\infty$; in particular, we have $\min_{Q_h^*\in F}\|Q_{h,k+1}- Q_h^*\|_\infty \leq \|Q_{h,k+1}- Q_h^*\|_\infty$)
\begin{equation}\label{dec08}
    \begin{aligned}
         &\min_{Q_h^*\in F}\|Q_{h,k+1}- Q_h^*\|_\infty-\min_{Q_h^*\in F}\|Q_{h,k}- Q_h^*\|_\infty\leq \|Q_{h,k+1}- Q_h^*\|_\infty-\|Q_{h,k}- Q_h^*\|_\infty\\
        &\leq -\rho^Q(1-e^{-\gamma h})\|Q_{h,k}- Q_h^*\|_\infty+\rho^Q h  \Big(L_f + \frac{L_p}{e^{\gamma h}-1} \|f\|_{\infty} \Big) \|\mu_k-\tilde \mu^*\|_{\text{TV}}\\ 
        &\leq -\rho^Q(1-e^{-\gamma h})\|Q_{h,k}- Q_h^*\|_\infty+\rho^Q h \Big(L_f + \frac{L_p}{e^{\gamma h}-1} \|f\|_{\infty} \Big)
         \Big(\|\mu_k - \tilde{\mu}_k\|_{\text{TV}}+ \|\tilde{\mu}_k - \tilde{\mu}^*\|_{\text{TV}}\Big)\\
        &\leq -\rho^Q(1-e^{-\gamma h})\|Q_{h,k}- Q_h^*\|_\infty\\
        &\quad+\rho^Q h \Big(L_f + \frac{L_p}{e^{\gamma h}-1} \|f\|_{\infty} \Big)
         \Big(\|\mu_k - \tilde{\mu}_k\|_{\text{TV}}+ \frac{L_Q}{2\beta-1-L_p} \|Q_{h,k}- Q_h^*\|_\infty\Big),
    \end{aligned}
\end{equation}
where in the last inequality we use \eqref{eq:mu_Q}.

\paragraph{Combined estimates}
Now we are ready to combine \eqref{dec07} and \eqref{dec08} together and obtain that
\begin{equation*}
    \begin{aligned}
 &\mathcal{L}_{k+1} -\mathcal{L}_k \\
 &\leq \Big(-W\rho^Q(1-e^{-\gamma h})+W\rho^Q h \Big(L_f + \frac{L_p\|f\|_{\infty}}{e^{\gamma h}-1}  \Big)\frac{L_Q}{2\beta-1-L_p} +\frac{2\Lambda_\mu L_Q}{2\beta-1-L_p }\Big)\|Q_{h,k}- Q_h^*\|_{\infty}\\
 &\quad- \Big(1-\Lambda_{\mu} -W\rho^Q h \Big(L_f + \frac{L_p\|f\|_{\infty} }{e^{\gamma h}-1} \Big)
    \Big) \|\mu_k-\tilde{\mu}_{k}\|_{\text{TV}}+\frac{\Lambda_\mu L_Q\rho^Q}{2\beta-1-L_p } \|\mathcal{T}(\cdot, \cdot)\|_{\infty}\\
    &\leq -c \mathcal{L}_k + e_k.
    \end{aligned}
\end{equation*}
% Given the relations 
% \begin{equation}\label{dec09}
%    W \rho^Q h \sqrt{N_{\mathcal{X}}}\big(L_f + \frac{1}{e^{\gamma h}-1} \|f\|_{L^{\infty}} L_p\big)< \Lambda_{\mu}^{1/2},\quad \text{and}\quad \Lambda_{\mu}^{1/2}<1/2
% \end{equation}
This inequality holds if we let 
\begin{equation}\label{dec_const}
\begin{aligned}
    &c_1: = \rho^Q\Big(1-e^{-\gamma h}-  \Big(L_f + \frac{L_p\|f\|_{\infty}}{e^{\gamma h}-1}  \Big)\frac{hL_Q}{2\beta-1-L_p} \Big)-\frac{2\Lambda_\mu L_Q}{W(2\beta-1-L_p) },\\
    &c_2:= 1-\Lambda_{\mu} -W\rho^Q h \big(L_f + \frac{L_p\|f\|_{\infty} }{e^{\gamma h}-1} \big),
    \end{aligned}
\end{equation}
and we require that
\begin{equation}\label{dec10}
    c: = \min\{c_1, c_2\} \in(0,1),
\end{equation}
with the error term denoted as
\begin{equation}
    e_k := \frac{\Lambda_\mu L_Q\rho^Q}{2\beta-1-L_p } \|\mathcal{T}(\cdot, \cdot)\|_{\infty}.
\end{equation}
Note that by \eqref{eqn:alg} and \eqref{eq:boundQ}, $\|\mathcal{T}(\cdot, \cdot)\|_{\infty}$ is bounded by
\begin{equation}
    \|\mathcal{T}(\cdot, \cdot)\|_{\infty}\leq h\|f\|_\infty+ (e^{-\gamma h}+1)\|Q_{h,k}\|_{\infty}\leq \frac{2h}{1-e^{-\gamma h}}\|f\|_{\infty}.
\end{equation}
Eventually we have the convergence
\begin{equation*}
\begin{aligned}
    \mathcal{L}_k &\leq (1-c)^k \mathcal{L}_0 + \frac{2h\Lambda_{\mu}L_Q\rho^Q \|f\|_{\infty}}{(1-e^{-\gamma h})(2\beta-1-L_p)} \sum_{j=0}^{k-1}(1-c)^{k-1-j} \\
    & \leq (1-c)^k \mathcal{L}_0  + \frac{2h\Lambda_{\mu}L_Q\rho^Q \|f\|_{\infty}}{c(1-e^{-\gamma h})(2\beta-1-L_p)}.  
    \end{aligned}
\end{equation*}
\end{proof}

% \begin{remark}
%     Recall that $\Lambda_{\mu} =1- \rho^{\mu}\big(1-2(1-\beta)-L_p\big)$. From \eqref{dec_const}-\eqref{dec10}, we observe that when $\rho^Q\ll \rho^{\mu}$, we need $W\sim 1/\rho^Q$ in order to ensure $c_1$ in \eqref{dec_const} is strictly positive, this implies that the convergence of $Q$ dominates the process (it is aligned with the fact that we have fast convergence for $\mu$ and slow convergence for $Q$). On the other hand, when $\rho^Q\gg \rho^{\mu}$, to ensure $c_2$ in \eqref{dec_const} is strictly positive, we require $W$ to be much smaller than $1$, so that  the convergence of $\mu$ dominates the process (it is aligned with the fact that we have fast convergence for $Q$ and slow convergence for $\mu$).
% \end{remark}

%%%%%%%%%%%%%%%%%%%%%%%%%%%%%%%%%%%%%%%%%%%%%%%%%%%

\section{Numerical experiment}\label{sec:numeric}
We carry out some numerical experiments to  validate our convergence result of \eqref{eqn:alg} with different ratios of fixed learning rates $\rho^{\mu}$ and $\rho^Q$. 
Our examples are adapted from those of \cite{angiuli2022unified} with slight modifications.
The algorithm we use is sample-based with stochastic approximations to the iterations~\eqref{eqn:alg}.
For better control of the numerical comparison, we use the maximum number of iterations $N_k$ as stopping criterion.

\begin{algorithm}[ht]
\caption{Unified Two-timescales Q-learning - Tabular version} 
\begin{algorithmic}[1]
\REQUIRE $T$ : number of time steps in a learning episode, 
\\ $\mathcal{{X}} =\{ x_0, \dots, x_{|\mathcal{{X}}|-1}\}$ : finite state space. \\ 
$\mathcal{{A}} =\{ a_0, \dots,
a_{|\mathcal{{A}}|-1}\}$ : finite action space. \\
$\mu_0$ : initial distribution of the representative player.\\ 
$\epsilon$ : parameter related to the $\epsilon$-greedy policy.\\
$N_k$ : number of episodes.\\
$\gamma, h:$ fixed constants.
\STATE \textbf{Initialization}: $Q^0(x,a)=0$ for all $(x,a)\in \mathcal{{X}}\times \mathcal{{A}}$,  $\mu^0_{n}=\left(\frac{1}{|\mathcal{{X}}|}, \dots,\frac{1}{|\mathcal{{X}}|}\right)$ for $n=0, \dots, T$ 

\FOR{each episode $k=1,2,\dots N_k$}
\STATE \textbf{Initialization:} Sample $X^k_{0} \sim \mu^{k-1}_{T}$ and set $Q^k \equiv Q^{k-1}$  
\FOR{$ n \gets 0$ to $T-1$} 
\STATE
\textbf{Update} $\mu$: \\
$\mu^{k}_{n} = \mu^{k-1}_{n} + \rho^{\mu} (\bm{\delta}(X^k_{n}) - \mu^{k-1}_{n})$ where $\bm{\delta}(X^k_{n}) = \left(\mathbf{1}_{x_0}(X^k_{n}), \dots, \mathbf{1}_{x_{|\mathcal{X}|-1}}(X^k_{n})\right) $
\STATE
\textbf{Choose action $A^k_{n}$} using the $\epsilon$-greedy policy derived from $Q^k(X^k_{n},\cdot)$  
\\
\textbf{Observe cost $f_{n+1}=f(X^k_{n},A^k_{n},\mu^{k}_{n})$} and state $X^k_{n+1}$ provided by the environment %

\STATE
\textbf{Update} $Q$:\\ $Q^k(X^k_{n},A^k_{n})= Q^k(X^k_{n},A^k_{n})+\rho^Q [h f_{n+1} +e^{-\gamma h} \min_{a'\in \mathcal{A}} Q^k(X^k_{n+1},a')-Q^k(X^k_{n},A^k_{n})]$
\ENDFOR
\ENDFOR
\RETURN $(\mu^k,Q^k)$
\end{algorithmic}
\end{algorithm}

\FloatBarrier

\paragraph{Benchmark problem}
For MFG and MFC problems introduced in \eqref{opt_prob-mfg} and \eqref{opt_prob-mfc},
we take $\mathcal{X}, \mathcal{A} \subset \mathbb{R}$, and define the cost function
\begin{equation}\label{eq: benchmark}
  f(x,\alpha,\mu) = \frac{1}{2} \alpha^2 + c_1 \left( x- c_2 m\right)^2 + c_3 \left( x- c_4 \right)^2 + c_5 m^2,
    \qquad
    b(x, \alpha,\mu) = \alpha,   
\end{equation}
where $m = \sum_{x \in \mathcal{X}} x \mu(x) $, 
$c_1=0.25$, $c_2=1.5$, $c_3=0.50$, $c_4=0.6$, $c_5=5$, 
%\JA{JA: need to replaced by $\gamma, h$}
 discount parameter $\gamma =1$ and volatility $\sigma = 0.3$. The infinite time horizon is truncated at time $T=20$. The continuous time is discretized using step $h= 0.01$.  
We adopt a larger action space $\mathcal{A} = \{ a_0=-2, \dots, a_{N_{\mathcal{A}}}=2 \}$ and the state space is $\mathcal{X} = \{ x_0=-2+x_c, \dots, x_{N_{\mathcal{X}}}=2+x_c \}$, where $x_c$ is the center of the state space.
The step size for the discretization of the state and action spaces  $\mathcal{X}$ and $\mathcal{A}$ is given by $\Delta = \sqrt{h} = 0.1$. 
For the discretization of the SDE $dX_t = \alpha_t dt + \sigma dB_t$, we consider the transition matrix given by  
\begin{equation}
   p(x'\mid x,a,\mu) \propto \mathbb{P}(Z \in [x'-\Delta/2, x'+\Delta/2])
\end{equation}
with $Z \sim \mathcal{N}(x+a, \sigma^2 h)$; the distribution is normalized to avoid any artifacts due to numerical approximations. 

We use the unified two timescale mean field Q-learning algorithm in \cite{angiuli2022unified} with a fixed ratio of step sizes $\rho^Q$ and $\rho^{\mu}$.
In the Q-learning, we set the number of episodes $N_k=140000$ and the learning rates $\rho^Q = 0.02, \rho^{\mu} = 0.0001$ (and hence ratio $\rho^Q / \rho^{\mu} = 200$) for the MFG problem,
$ \rho^Q = 0.0001, \rho^{\mu} = 0.5$ (ratio $\rho^Q / \rho^{\mu} = 0.0002$) for the MFC problem.

\paragraph{Results}
We compare the numerical value functions achieved by the two-timescale Q-learning algorithm  with the calculated theoretical value functions: Figure~\ref{fig:a} plots value functions of MFG and Figure~\ref{fig:b} plots value functions of MFC problem.
One can calculate theoretical value functions from the HJB equations based on Theorem~\ref{thm:V-HJB}, and we refer computation details to  \cite[Appendix A]{angiuli2022unified}.
In addition, we present the optimal control function $\hat\alpha = \arg\min_{a} Q(x,a)$ and the theoretical optimal control function in Figure~\ref{fig:c} for MFG and in Figure~\ref{fig:d} for MFC problem.
Figure~\ref{fig:e} shows the empirical equilibrium distribution averaged over last $10000$ episodes in the unified two-timescale Q-learning algorithm with different ratios of learning rates $\rho^{\mu}, \rho^{Q}$.

\begin{figure}[!ht]
\centering
	\begin{subfigure}[t]{0.32\textwidth}
		\centering
		\includegraphics[width=1\textwidth]{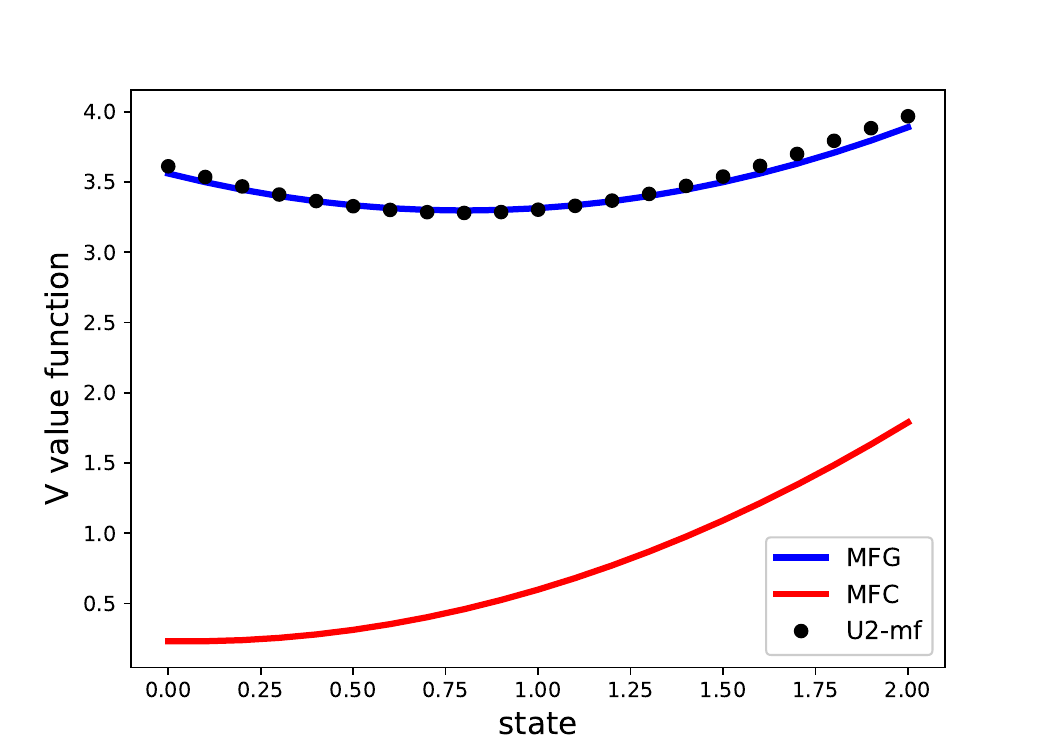}
		\caption{MFG Value function}
		\label{fig:a}
	\end{subfigure}
	\begin{subfigure}[t]{0.32\textwidth}
		\centering
		\includegraphics[width=1\textwidth]{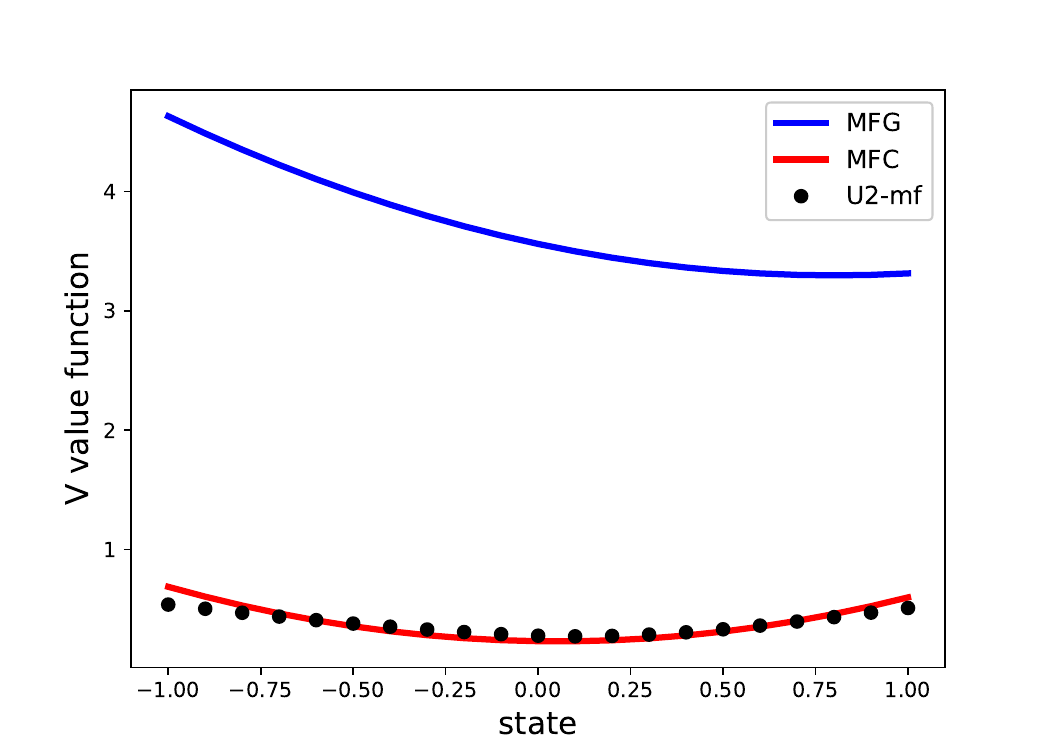}
		\caption{MFC Value function}
		\label{fig:b}
	\end{subfigure} 

 \begin{subfigure}[t]{0.32\textwidth}
		\centering
		\includegraphics[width=1\textwidth]{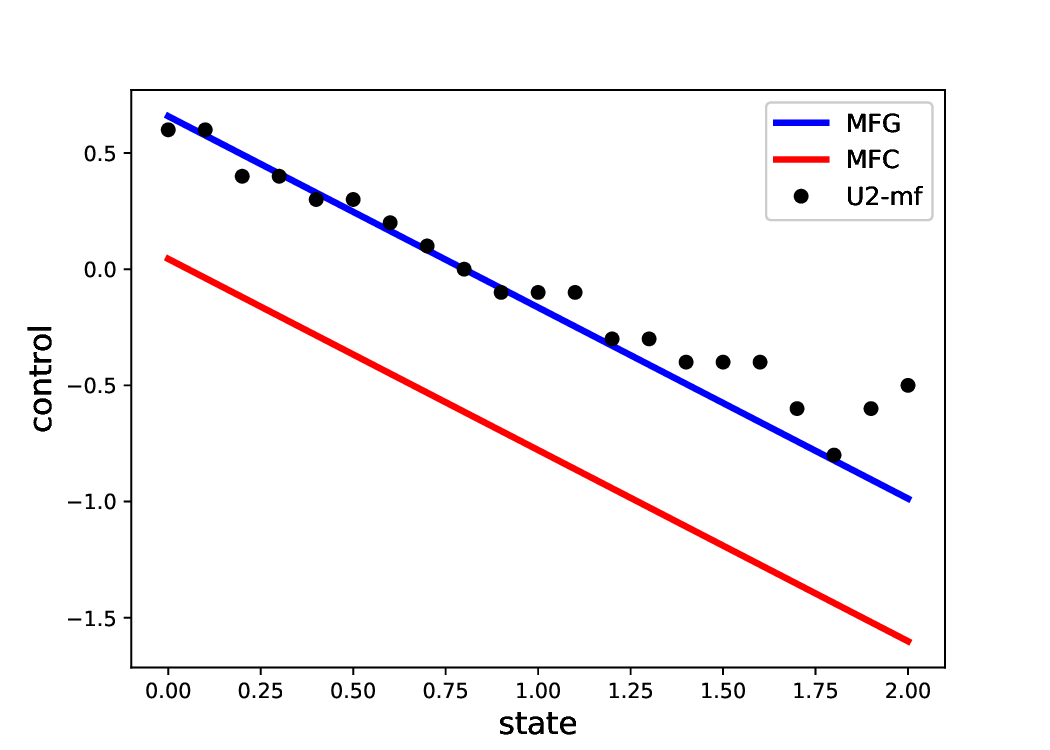}
		\caption{MFG control averaged over last 10000 episodes}
		\label{fig:c}
	\end{subfigure}
	\begin{subfigure}[t]{0.32\textwidth}
		\centering
		\includegraphics[width=1\textwidth]{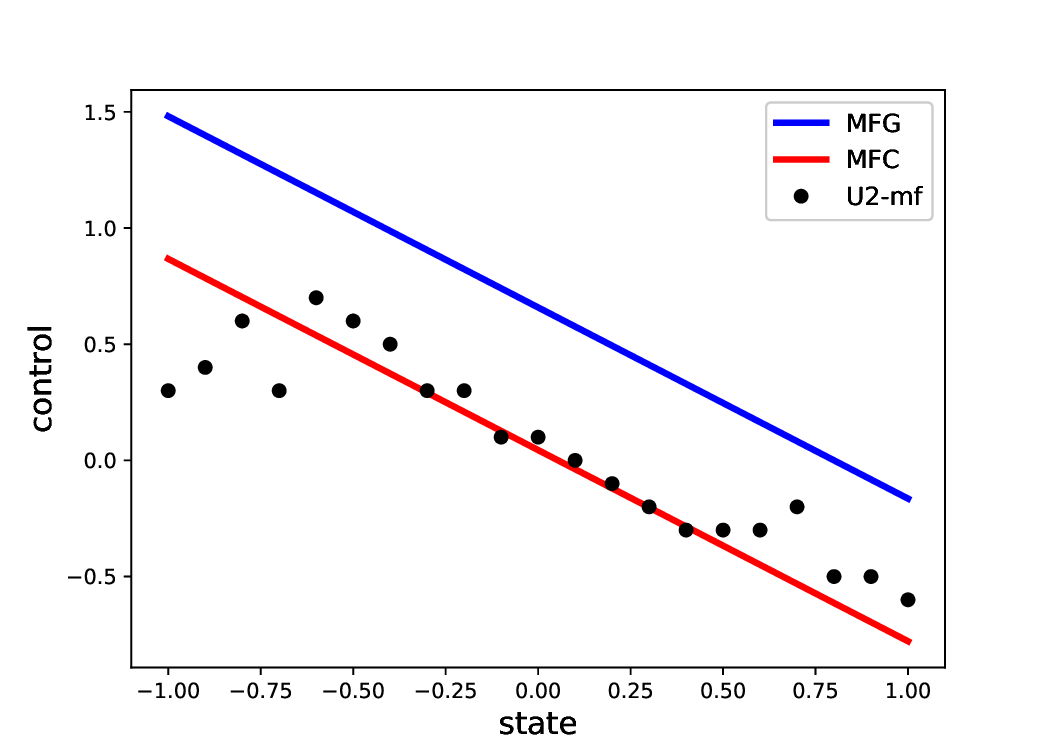}
		\caption{MFC control averaged over last 10000 episodes}
		\label{fig:d}
	\end{subfigure}	
        \begin{subfigure}[t]{0.32\textwidth}
		\centering
		\includegraphics[width= 1\textwidth]{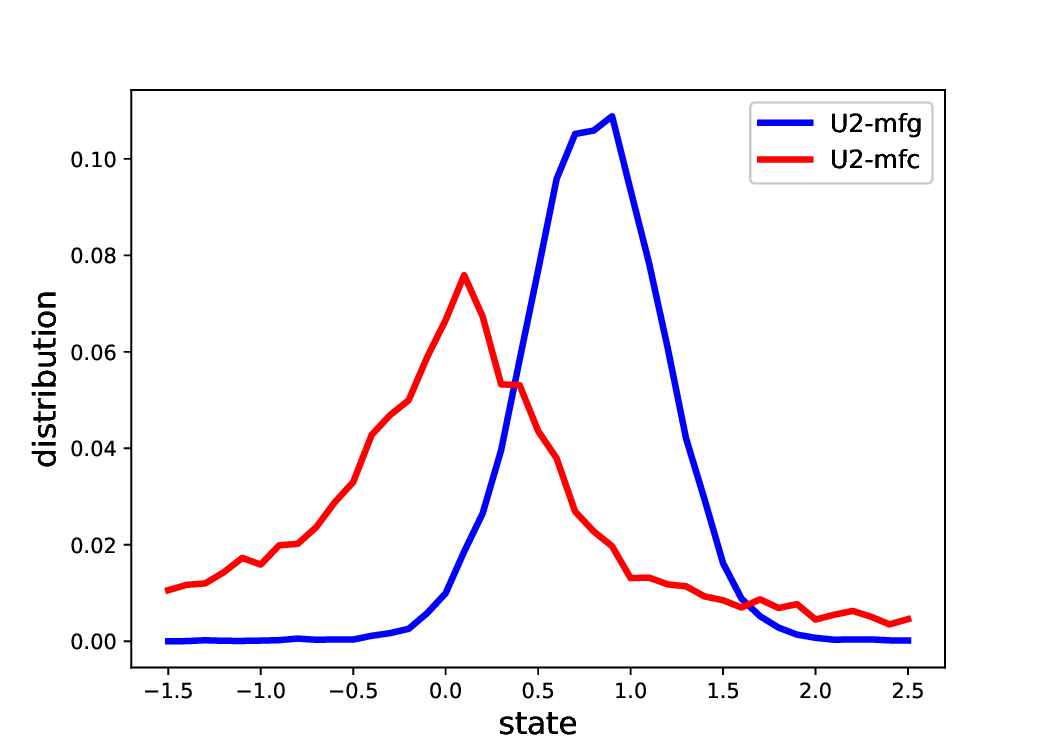}
		\caption{MFG and MFC distribution averaged over last 10000 episodes}
		\label{fig:e}
	\end{subfigure}

	\caption{Numerical results of the  two-timescale Q-learning algorithm where MFG learning rates are $\rho^Q = 0.02, \rho^{\mu} = 0.0001$, and
 MFC learning rates are $ \rho^Q = 0.0001, \rho^{\mu} = 0.5$. The theoretical value/control functions are represented by solid lines and numerical value/control functions are represented by dotted lines in figure~\ref{fig:a}-\ref{fig:d}.}	
\label{fig:compare}
\end{figure}

\paragraph{Intermediate ratios of $\rho^{Q}/\rho^{\mu}$} 
In addition to extreme ratios where numerically $\rho^{Q}/\rho^{\mu} =200$ for MFG problem and $\rho^{Q}/\rho^{\mu} =0.0002$ for MFC problem, we take some intermediate ratios where $ \rho^{Q}/\rho^{\mu} =10, 1, 0.1 $ and present the respective resulting value functions in Figure~\ref{fig:compareV}.
In the figure, the theoretical solutions are labeled ``MFG'' and ``MFC'' and represented by solid lines, and the numerical results of two timescale algorithm are labeled with prefix ``U2-'' and represented by dotted lines.
We observe in the intermediate regimes, the algorithms seem to converge to some solutions lying between the MFG and MFC value functions; while in this work we do not identify these limits, this would be an interesting future research direction.

\begin{figure}[htb]
		\centering
		\includegraphics[width= 8cm]{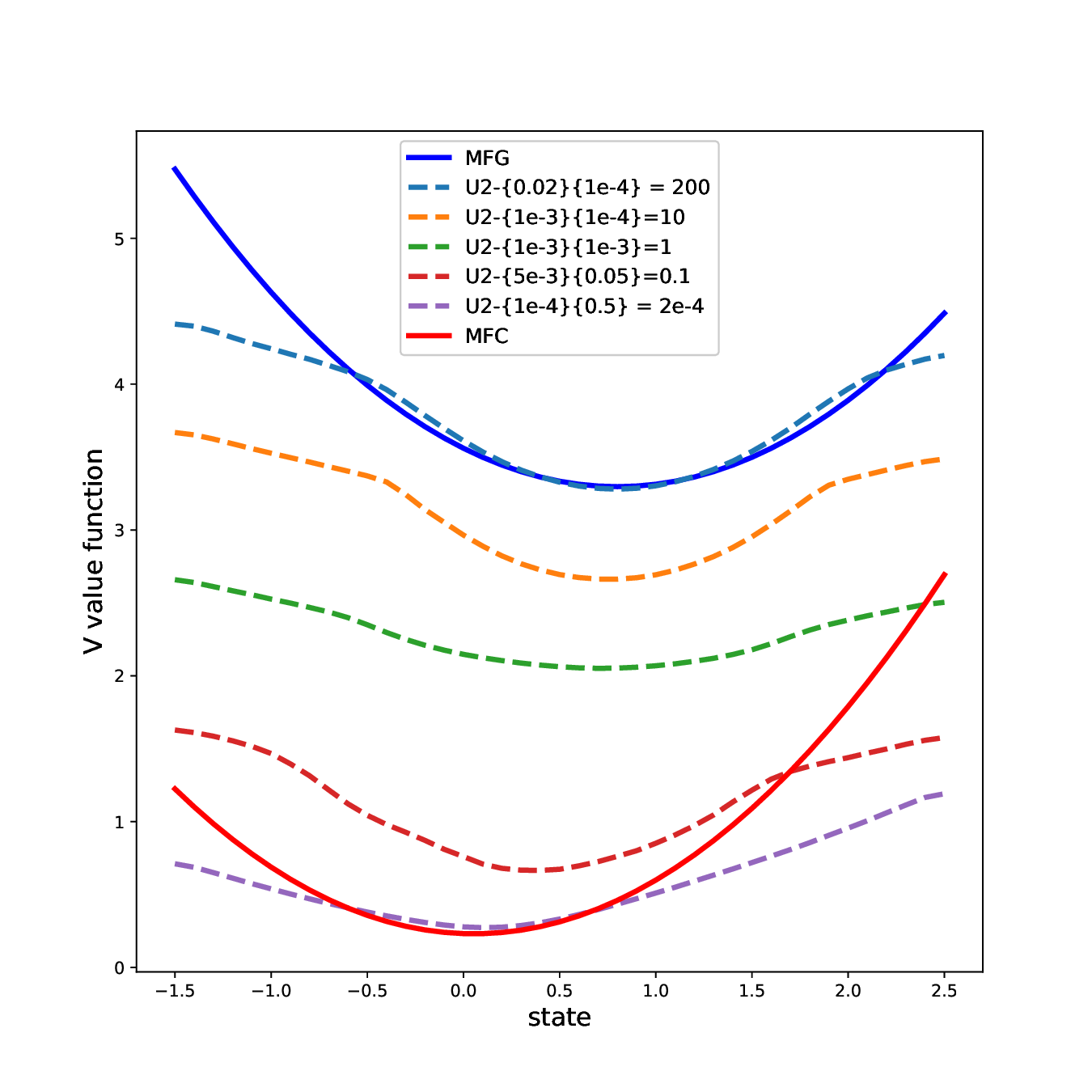}
         \caption{Numerical results of the  two-timescale Q-learning algorithm where intermediate ratios $\rho^{Q}/\rho^{\mu} =10, 1, 0.1$ are adopted, in addition to numerically extreme ratios $\rho^{Q}/\rho^{\mu} =200, 0.0002$.}
           \label{fig:compareV}
\end{figure}

\FloatBarrier

\section{Conclusion}\label{sec:conclusion}
% we conclude the paper by summarizing the key contributions, discussing the implications of our findings, and suggesting potential avenues for future research.
% In this paper, we present a comprehensive convergence analysis of the unified two timescales mean field Q-learning algorithm, shedding light on its ability to address both mean field game (MFG) and mean field control (MFC) problems by adjusting the ratio between the two learning steps. 
% We have revealed the underlying mechanism behind the production of two different mean field solutions through the single timescale convergence of the two timescale algorithm. 
% Additionally, our research has established that the optimal value functions in MFG and MFC problems adhere to different Hamilton-Jacobi-Bellman equations. 
% These findings not only highlight the differences between MFG and MFC problems but also contribute to a deeper understanding of the algorithm's behavior.
% Moving forward, this work opens up avenues for further exploration and refinement in the field of mean field reinforcement learning.

In this work, by establishing the approximation diagram  Fig.~\ref{fig:roadmap}, we explain why the two-timescale Q-learning algorithm can converge to MFG or MFC solutions by tuning two learning rates.  Based on our constructed Lyapunov function, we provide a novel unified convergence result for the algorithm for all ranges of learning rate ratios. It would be interesting to investigate what type of problems that the two-timescale Q-learning algorithm solves when $0<\rho^Q/\rho^\mu<\infty$, as shown in Figure~\ref{fig:compareV}. We guess that for the intermediate regime, devising a mixed model of MFG and MFC might be a reasonable approach, and we leave it as our future work. We believe that the idea of this Lyapunov function construction can shed lights on convergence proofs for other algorithms in the study of MFC and MFG. 
\clearpage
\appendix
\section{}
\label{sec:appd}
The case of MFC problems need extra treatments due to the value function's dependence on the changing population distribution. We thus consult the It\^{o}-Lions' formula in Wasserstein space studied in \cite{buckdahn2017mean}, and give a brief review of additional required assumptions in order to apply this It\^{o}-Lions' formula for MFC. 

Consider the square-integrable space $\mathcal{P}(\mathcal{X})$, the lifting of functions $u: \mathcal{P}(\mathcal{X})\to \Rm$ is defined as $\tilde{u}(\xi):= u(P[\xi])$. We say that $u$ is differentiable (resp., $C^1$) on $\mathcal{P}(\mathcal{X})$ if the lift $\tilde{u}$ is Fr\'{e}chet differentiable on $L^2(\mathcal{F}; \mathcal{X})$, that is, there exists a linear continuous mapping $D\tilde{u}(\xi):L^2(\mathcal{F}; \mathcal{X})\to \Rm$ such that 
\begin{equation}
    \tilde u(\xi+\eta) - \tilde u(\xi) = D\tilde u(\xi)(\eta)+\lito(\|\eta\|),
\end{equation}
with $\|\eta\|\to 0$ for $\eta\in L^2(\mathcal{F}; \mathcal{X})$. On the law $P[\xi]$, for $\xi, \xi' \in L^2(\mathcal{F}; \mathcal{X})$, one can write
\begin{equation}
    u(P[\xi'])-u(P[\xi]) =\mathbb{E}[\d_{\mu} u(P[\xi], \xi)\cdot (\xi'-\xi)]+\lito(\|\xi'-\xi\|)
\end{equation}
to define $\d_\mu u$. Moreover, the second derivative is defined as
\begin{equation}
    \d_{\mu}^2 u (\mu,x,y): = \big(\d_{\mu}\big((\d_\mu u)_j(\cdot, y)\big)(\mu,x)\big)_{1\leq j\leq d},\quad \text{for}~(\mu,x,y)\in \mathcal{P}(\mathcal{X})\times \mathcal{X}\times \mathcal{X}.
\end{equation}
Let us state the expansion formula from \cite{buckdahn2017mean}:
\begin{lemma}[\cite{buckdahn2017mean}, Lemma 2.1]\label{lem:Ito}
If $(\d_\mu u)_j(\cdot, y)\in C_b^{1,1}(\mathcal{P}(\mathcal{X}))$ for all $y\in\mathcal{X}, 1\leq j\leq d$, $\d_\mu u(\mu,\cdot)$ is differentiable for every $\mu\in\mathcal{P}(\mathcal{X})$, and $\d_{\mu}^2 u, \d_y\d_\mu u $ are bounded and Lipschitz continuous,  then one has the second-order expansion
\begin{equation}
\begin{aligned}
    u(P[\xi'])-u(P[\xi])&= \mathbb{E}[\d_\mu u(P[\xi], \xi) \cdot \eta]+\frac{1}{2} \mathbb{E}\Big[\tilde{\mathbb{E}}\big[\Tr\big(\d_{\mu}^2 u(P[\xi], \tilde{\xi}, \xi)\cdot \tilde{\eta}\otimes\eta\big)\big]\Big]\\
    &\quad +\frac{1}{2}\mathbb{E}\Big[\Tr\big(\d_y \d_{\mu} u(P[\xi], \xi)\cdot \eta\otimes\eta\big) \Big] + O(\|\eta\|^3)
\end{aligned}
\end{equation}
where $\eta = \xi'-\xi$, and $\tilde{\mathbb{E}}[\cdot]=\int_{\tilde{\mathcal{X}}}(\cdot) d\tilde{P}$ associated with the copy $(\tilde{\mathcal{X}}, \tilde{\mathcal{F}}, \tilde{P})$ where $\tilde{P}[\tilde{\xi}] = P[\xi]$.
\end{lemma}

Now we are ready to restate the Theorem~\ref{thm:conv_h} with complete assumptions.
\begin{thm}\label{thm:conv_formal} We assume that $\alpha$ is Lipschitz in $x$: there exists a constant $C_{\alpha}>0$ such that
\begin{equation}
    \|\alpha(x_1)-\alpha(x_2)\|\leq C_{\alpha}\|x_1-x_2\|.
\end{equation}
    With Assumptions~\ref{assump:f}, \ref{assump:f_extra}, and \ref{assump:b-sigma}, in addition to  assumptions on $f$ in order to apply Lemma \ref{lem:Ito}, we have the approximations that, 
for all $x \in {\mathcal{X}}$,
\begin{equation}
\begin{aligned}
&V_{h, \text{MFG}}^{\mu,\alpha}(x) = V_{\text{MFG}}^{\mu,\alpha}(x) + \bigo(h^{1/2}), \\% = V_{\text{MFG}}^{\mu,\alpha}(x) + \lito(1),\\
&V_{h, \text{MFC}}^{\alpha} (x)  = V_{\text{MFC}}^{\alpha}(x) + \bigo(h^{1/2}), % = V_{\text{MFC}}^{\alpha}(x)+\lito(1),
\end{aligned}
\end{equation}
when $h\to 0$.
\end{thm}
\begin{proof}
   We may consider the discrete time iteration
    \begin{equation}
        X_{h}^{k+1} = b(X_h^k, \alpha(X_{h}^k)) h + \sigma(X_h^k, \alpha(X_{h}^k)) \sqrt{h} B_k,
    \end{equation}
    with $B_k\sim_{i.i.d} \mathcal{N}(0,I)$, which can be viewed as 
  the Euler-Maruyama scheme of the SDE
    \begin{equation}
		dX_t = b(X_t, \alpha_t)dt + \sigma(X_t, \alpha_t) dB_t.
	\end{equation}
 We ignore $b, \sigma$'s dependence on $\mu$ here by assuming the limiting distribution is fixed. It is well-known that the Euler-Maruyama scheme is an order $1/2$-scheme in the strong sense \cite{kloeden2012numerical}. In other words, with Lipschitzness assumptions of $b, \sigma$, and a slight modification of the induction proof, one can conclude immediately that there exists a constant $C>0$ such that
 \begin{equation}
     \mathbb{E}\big[\|X_h^k-X_{kh}\|\big|X_0=x\big]\leq C h^{1/2}
 \end{equation}
 for all $k\geq 1$. Then, for $t\in [kh, (k+1)h)$, since
 \begin{equation}
     X_t = X_{kh}+\int_{kh}^t b(X_s, \alpha_s ) ds+ \int_{kh}^t \sigma(X_s, \alpha_s ) dB_s,
 \end{equation}
 by It\^{o}'s isometry and boundedness of $b, \sigma$, we get
 \begin{equation}
     \mathbb{E}\big[\|X_t-X_{kh}\|\big|X_0=x\big]\leq Ch^{1/2}.
 \end{equation}
 Therefore,  the triangle inequality gives that for $t\in [kh, (k+1)h)$,
 \begin{equation}
     \mathbb{E}\big[\|X_h^k-X_{t}\|\big|X_0=x\big]\leq C h^{1/2}.
 \end{equation}
 \paragraph{MFG:} We start from the definition of $V_{h, \text{MFG}}^{\mu,\alpha}$  and derive that
 \begin{equation}
\begin{aligned}
   V_{h, \text{MFG}}^{\mu,\alpha}(x) 
   &= \mathbb{E} \left[ h \sum_{k=0}^{\infty}  e^{-\gamma k h} f(X_h^k , \alpha(X_h^k), \mu )  \bigg\vert X_0 =x  \right] \\
  =&
  \mathbb{E} \left[
  \sum_{k=0}^\infty  
  \int_{kh}^{(k+1)h} e^{-\gamma k h}   f(X_h^k , \alpha(X_h^k), \mu )   ds
  \bigg\vert X_0 =x  \right] \\
  =& 
  \mathbb{E} \left[\int_0^\infty e^{-\gamma s} f(X_s, \alpha(X_s), \mu) ds\bigg\vert X_0=x\right]\\
  &
  \quad + \mathbb{E} \left[
  \sum_{k=0}^\infty \int_{kh}^{(k+1)h} \Big(e^{-\gamma k h}   f(X_h^k , \alpha(X_h^k), \mu )  - e^{-\gamma s} f(X_s, \alpha(X_s), \mu) \Big)ds\bigg\vert X_0=x\right]\\
  &:= V_{\text{MFG}}^{\mu,\alpha}(x) + \mathcal{E}_{\text{MFG}},
  \end{aligned}
\end{equation}
and what remains is to estimate the error term $\mathcal{E}_{\text{MFG}}$. Apply the triangle inequality, we get
\begin{equation}
    \begin{aligned}
\big|\mathcal{E}_{\text{MFG}}\big|&\leq \mathbb{E} \left[
  \sum_{k=0}^\infty \int_{kh}^{(k+1)h} \big|e^{-\gamma k h}-e^{-\gamma s}\big|   \big|f(X_h^k , \alpha(X_h^k), \mu ) \big|ds\bigg\vert X_0=x\right]\\
  &\quad + \mathbb{E} \left[
  \sum_{k=0}^\infty \int_{kh}^{(k+1)h} e^{-\gamma s}  \Big| f(X_h^k , \alpha(X_h^k), \mu )  - f(X_s, \alpha(X_s), \mu)\Big| ds\bigg\vert X_0=x\right]\\
  &\leq \|f\|_{\infty} \sum_{k=0}^\infty \int_{kh}^{(k+1)h}\big| e^{-\gamma k h}-e^{-\gamma s} \big|ds+ \tilde{L}\mathbb{E} \left[
  \sum_{k=0}^\infty \int_{kh}^{(k+1)h} e^{-\gamma s}  \|X_h^k-X_s\| ds\bigg\vert X_0=x\right]\\
  & = \|f\|_{\infty}\bigo(h)\int_0^\infty e^{-\gamma s} ds + \tilde{L} \bigo(h^{1/2})\int_0^\infty e^{-\gamma s} ds = \bigo(h^{1/2})
    \end{aligned}
\end{equation}
for small $h$, where in the last inequality we use Assumption~\ref{assump:f}, Lipschitz assumption of $\alpha$, and take $\tilde{L}:=\max\{K_x, K_{\alpha}C_{\alpha}\}$. Thus we conclude that
\begin{equation}
     V_{h, \text{MFG}}^{\mu,\alpha}(x)  =V_{\text{MFG}}^{\mu,\alpha}(x) + \lito(1),
\end{equation}
as $h\to 0$.
 \paragraph{MFC:} For the MFC case, we need to additionally deal with $f(x,a,\mu)$'s dependence on the changing distributions $\mu$. We again start from the definition of $V_{h, \text{MFC}}^{\alpha}$  and derive that

 \begin{equation}
\begin{aligned}
   V_{h, \text{MFC}}^{\alpha}(x) 
   &= \mathbb{E} \left[ h \sum_{k=0}^{\infty}  e^{-\gamma k h} f(X_h^k , \alpha(X_h^k), \mathcal{P}[X_h^k] )  \bigg\vert X_0 =x  \right] \\
  =&
  \mathbb{E} \left[
  \sum_{k=0}^\infty  
  \int_{kh}^{(k+1)h} e^{-\gamma k h}   f(X_h^k , \alpha(X_h^k), \mathcal{P}[X_h^k])   ds
  \bigg\vert X_0 =x  \right] \\
  =& 
  \mathbb{E} \left[\int_0^\infty e^{-\gamma s} f(X_s, \alpha(X_s), \mathcal{P}[X_s]) ds\bigg\vert X_0=x\right]\\
  &
  \quad + \mathbb{E} \left[
  \sum_{k=0}^\infty \int_{kh}^{(k+1)h} \Big(e^{-\gamma k h}   f(X_h^k , \alpha(X_h^k), \mathcal{P}[X_h^k] )  - e^{-\gamma s} f(X_s, \alpha(X_s), \mathcal{P}[X_s]) \Big)ds\bigg\vert X_0=x\right]\\
  &:= V_{\text{MFC}}^{\alpha}(x) + \mathcal{E}_{\text{MFC}}.
  \end{aligned}
\end{equation}
The error term $\mathcal{E}_{\text{MFC}}$ can be further split into
\begin{equation}
    \begin{aligned}
   \big|\mathcal{E}_{\text{MFC}}\big|     &\leq \mathbb{E} \left[
  \sum_{k=0}^\infty \int_{kh}^{(k+1)h} \big|e^{-\gamma k h}-e^{-\gamma s}\big|   \big|f(X_h^k , \alpha(X_h^k), \mathcal{P}[X_h^k] ) \big|ds\bigg\vert X_0=x\right]\\
  &\quad + \mathbb{E} \left[
  \sum_{k=0}^\infty \int_{kh}^{(k+1)h} e^{-\gamma s}  \Big| f(X_h^k , \alpha(X_h^k), \mathcal{P}[X_h^k] )  - f(X_s, \alpha(X_s), \mathcal{P}[X_h^k])\Big| ds\bigg\vert X_0=x\right]\\
  &\quad + \mathbb{E} \left[
  \sum_{k=0}^\infty \int_{kh}^{(k+1)h} e^{-\gamma s}  \Big| f(X_s , \alpha(X_s), \mathcal{P}[X_h^k] )  - f(X_s, \alpha(X_s), \mathcal{P}[X_s])\Big| ds\bigg\vert X_0=x\right]\\
  &:= I+II+III.
    \end{aligned}
\end{equation}
Apply the triangle inequality and Assumption~\ref{assump:f} as we did for MFG, we get again that
\begin{equation}
    I+II = O(h^{1/2}).
\end{equation}
For $III$, we need to apply Lemma \ref{lem:Ito}. By taking $\eta_{k,s}:= X_h^k-X_s$ and write $f(X_s, \alpha(X_s), \cdot)\equiv f_s(\cdot)$, we have that
\begin{equation}
\begin{aligned}
    III&=  \mathbb{E} \left[
  \sum_{k=0}^\infty \int_{kh}^{(k+1)h} e^{-\gamma s} \mathbb{E}[\d_\mu f_s(\mathcal{P}[X_s], X_s) \cdot \eta_{k,s}]ds\bigg\vert X_0=x\right]\\
    &\quad +\mathbb{E} \left[
  \sum_{k=0}^\infty \int_{kh}^{(k+1)h} e^{-\gamma s} \frac{1}{2} \mathbb{E}\Big[\tilde{\mathbb{E}}\big[\Tr\big(\d_{\mu}^2f_s(\mathcal{P}[X_s], \tilde{X_s}, X_s)\cdot \tilde{\eta}_{k,s}\otimes\eta_{k,s}\big)\big]\Big]ds\bigg\vert X_0=x\right]\\
    &\quad +\frac{1}{2} \mathbb{E} \left[
  \sum_{k=0}^\infty \int_{kh}^{(k+1)h} e^{-\gamma s}\mathbb{E}\Big[\Tr\big(\d_y \d_{\mu} V^{\alpha }_{h, MFC}(\mathcal{P}[x], x)\cdot \eta_{k,s}\otimes\eta_{k,s}\big) \Big]ds\bigg\vert X_0=x\right]\\
  &\quad +  \mathbb{E} \left[
  \sum_{k=0}^\infty \int_{kh}^{(k+1)h} e^{-\gamma s}O(\|\eta_{k,s}\|^3)ds\bigg\vert X_0=x\right] = O(h^{1/2}),
 \end{aligned}
\end{equation}
since the first term above dominates others as $\mathbb{E}[\|\eta_{k,s}\|\vert X_0=x]\sim O(h^{1/2})$ for small $h$. Thus we conclude that
\begin{equation}
     V_{h, \text{MFC}}^{\mu,\alpha}(x)  =V_{\text{MFC}}^{\mu,\alpha}(x) + \lito(1),
\end{equation}
as $h\to 0$.
\end{proof}

\section{}\label{sec:uniqueness}
\begin{lemma}
    With Assumptions \ref{assump:f}, \ref{assump:p} and \ref{assump:doeblin}, there exists a pair of fixed points $(Q_h^*, \tilde\mu^*)$ solving \eqref{eqn:conv_bellman} and \eqref{eqn:conv_tildemu}. In addition, if
    \begin{equation}
    \frac{L_Q}{2\beta-1-L_p}  \frac{hL_f+e^{-\gamma h}L_p\|Q\|_\infty}{1-e^{-\gamma h}}<1,
\end{equation}
then such a pair of fixed points is unique.
\end{lemma}
\begin{proof}
 We have shown the existence in the proof of Proposition \ref{prop:brouwer}. In terms of uniqueness, suppose that we find two fixed points $(Q_{h,1}^*, \tilde\mu_1^*)$ and $(Q_{h,2}^*, \tilde\mu_2^*)$ both solving \eqref{eqn:conv_bellman} and \eqref{eqn:conv_tildemu}, by Lemma \ref{lem:mu_and_Q}, we have 
\begin{equation}
    \|\tilde{\mu}_1^*-\tilde{\mu}_2^*\|_{\text{TV}} \leq \frac{L_Q}{2\beta-1-L_p} \|Q_{h,1}^*- Q_{h,2}^*\|_\infty.
\end{equation}
On the other hand, utilizing the Bellman operator $\mathcal{B}_{\mu}(Q)$ defined in (\ref{Bellman_operator}), we get 
\begin{equation}
\begin{aligned}\label{mar26_lip}
    \|Q_{h,1}^*- Q_{h,2}^*\|_\infty &= \|\mathcal{B}_{\tilde{\mu}_1^*} (Q_{h,1}^*)-\mathcal{B}_{\tilde{\mu}_2^*} (Q_{h,2}^*)\|_\infty\\
    &\leq (hL_f+e^{-\gamma h}L_p\|Q\|_\infty)\|\tilde{\mu}_1^*-\tilde{\mu}_2^*\|_{\text{TV}}+e^{-\gamma h} \|Q_{h,1}^*- Q_{h,2}^*\|_\infty,
    \end{aligned}
\end{equation}
so that 
\begin{equation}
    \|\tilde{\mu}_1^*-\tilde{\mu}_2^*\|_{\text{TV}} \leq \frac{L_Q}{2\beta-1-L_p}  \frac{hL_f+e^{-\gamma h}L_p\|Q\|_\infty}{1-e^{-\gamma h}}\|\tilde{\mu}_1^*-\tilde{\mu}_2^*\|_{\text{TV}}.
\end{equation}
With sufficiently small $h$ and large $\gamma$ such that $\gamma h \gg 1$, we have the factor
\begin{equation}
    \frac{L_Q}{2\beta-1-L_p}  \frac{hL_f+e^{-\gamma h}L_p\|Q\|_\infty}{1-e^{-\gamma h}}<1,
\end{equation}
and therefore $\tilde{\mu}_1^*=\tilde{\mu}_2^*$ in the total variation norm, which implies that $Q_{h,1}^*= Q_{h,2}^*$.
\end{proof}

%%%%%%%%%%%%%%%%%%%%%%%%%%%%%%%%%%%%%%%%%%%%%%%%

%\bibliographystyle{abbrv} 
\bibliography{references}

\begin{thebibliography}{53}
\providecommand{\natexlab}[1]{#1}
\providecommand{\url}[1]{\texttt{#1}}
\expandafter\ifx\csname urlstyle\endcsname\relax
  \providecommand{\doi}[1]{doi: #1}\else
  \providecommand{\doi}{doi: \begingroup \urlstyle{rm}\Url}\fi

\bibitem[Angiuli et~al.(2022)Angiuli, Fouque, and
  Lauri{\`e}re]{angiuli2022unified}
Andrea Angiuli, Jean-Pierre Fouque, and Mathieu Lauri{\`e}re.
\newblock {Unified reinforcement Q-learning for mean field game and control
  problems}.
\newblock \emph{Mathematics of Control, Signals, and Systems}, 34\penalty0
  (2):\penalty0 217--271, 2022.

\bibitem[Angiuli et~al.(2023)Angiuli, Fouque, Lauri{\`e}re, and
  Zhang]{angiuli2023convergence}
Andrea Angiuli, Jean-Pierre Fouque, Mathieu Lauri{\`e}re, and Mengrui Zhang.
\newblock Convergence of multi-scale reinforcement q-learning algorithms for
  mean field game and control problems.
\newblock \emph{arXiv preprint arXiv:2312.06659}, 2023.

\bibitem[Angiulia et~al.(2023)Angiulia, Fouquea, and
  Lauri{\`e}reb]{angiulia2023reinforcement}
Andrea Angiulia, Jean-Pierre Fouquea, and Mathieu Lauri{\`e}reb.
\newblock {Reinforcement Learning for Mean Field Games, with Applications to
  Economics}.
\newblock \emph{Machine Learning and Data Sciences for Financial Markets: A
  Guide to Contemporary Practices}, page 393, 2023.

\bibitem[Bellman(1957)]{bellman1957markovian}
Richard Bellman.
\newblock A markovian decision process.
\newblock \emph{Journal of mathematics and mechanics}, pages 679--684, 1957.

\bibitem[Bensoussan et~al.(2013)Bensoussan, Yam, and Frehse]{Bensoussan2013}
Alain Bensoussan, Phillip Yam, and Jens Frehse.
\newblock \emph{{Mean Field Games and Mean Field Type Control Theory}}.
\newblock {SpringerBriefs in Mathematics}. Springer, 2013.
\newblock ISBN 978-1-4614-8507-0.
\newblock \doi{10.1007/978-1-4614-8508-7}.

\bibitem[Bertsekas(2019)]{bertsekas2019reinforcement}
Dimitri Bertsekas.
\newblock \emph{Reinforcement learning and optimal control}.
\newblock Athena Scientific, 2019.

\bibitem[Bhandari and Russo(2024)]{bhandari2024global}
Jalaj Bhandari and Daniel Russo.
\newblock Global optimality guarantees for policy gradient methods.
\newblock \emph{Operations Research}, 2024.

\bibitem[Borkar(1997)]{borkar1997stochastic}
Vivek~S Borkar.
\newblock Stochastic approximation with two time scales.
\newblock \emph{Systems \& Control Letters}, 29\penalty0 (5):\penalty0
  291--294, 1997.

\bibitem[Borkar(2008)]{borkar2008stochastic}
V.S. Borkar.
\newblock \emph{{Stochastic Approximation: A Dynamical Systems Viewpoint}}.
\newblock Cambridge University Press, 2008.
\newblock ISBN 9780521515924.
\newblock URL \url{https://books.google.com.hk/books?id=QLxIvgAACAAJ}.

\bibitem[Buckdahn et~al.(2017)Buckdahn, Li, Peng, and Rainer]{buckdahn2017mean}
Rainer Buckdahn, Juan Li, Shige Peng, and Catherine Rainer.
\newblock Mean-field stochastic differential equations and associated pdes.
\newblock \emph{Annals of probability}, 45\penalty0 (2):\penalty0 824--878,
  2017.

\bibitem[Busoniu et~al.(2008)Busoniu, Babuska, and
  De~Schutter]{busoniu2008comprehensive}
Lucian Busoniu, Robert Babuska, and Bart De~Schutter.
\newblock A comprehensive survey of multiagent reinforcement learning.
\newblock \emph{IEEE Transactions on Systems, Man, and Cybernetics, Part C
  (Applications and Reviews)}, 38\penalty0 (2):\penalty0 156--172, 2008.

\bibitem[Caines et~al.(2006)Caines, Huang, and Malham{\'e}]{caines2006large}
Peter~E Caines, Minyi Huang, and Roland~P Malham{\'e}.
\newblock Large population stochastic dynamic games: closed-loop mckean-vlasov
  systems and the nash certainty equivalence principle.
\newblock \emph{Communications in Information and Systems}, 6\penalty0
  (3):\penalty0 221--252, 2006.

\bibitem[Carmona and Delarue(2018)]{carmona2018probabilistic}
Ren{\'e} Carmona and Fran{\c{c}}ois Delarue.
\newblock \emph{{Probabilistic theory of mean field games with applications
  I-II}}.
\newblock Springer, 2018.

\bibitem[Carmona et~al.(2019)Carmona, Lauri{\`e}re, and Tan]{carmona2019linear}
Ren{\'e} Carmona, Mathieu Lauri{\`e}re, and Zongjun Tan.
\newblock Linear-quadratic mean-field reinforcement learning: convergence of
  policy gradient methods.
\newblock \emph{arXiv preprint arXiv:1910.04295}, 2019.

\bibitem[Carmona et~al.(2023)Carmona, Lauri{\`e}re, and Tan]{carmona2023model}
Ren{\'e} Carmona, Mathieu Lauri{\`e}re, and Zongjun Tan.
\newblock Model-free mean-field reinforcement learning: mean-field mdp and
  mean-field q-learning.
\newblock \emph{The Annals of Applied Probability}, 33\penalty0 (6B):\penalty0
  5334--5381, 2023.

\bibitem[Cui and Koeppl(2021)]{cui2021approximately}
Kai Cui and Heinz Koeppl.
\newblock Approximately solving mean field games via entropy-regularized deep
  reinforcement learning.
\newblock In \emph{International Conference on Artificial Intelligence and
  Statistics}, pages 1909--1917. PMLR, 2021.

\bibitem[Fu et~al.(2019)Fu, Yang, Chen, and Wang]{fu2019actor}
Zuyue Fu, Zhuoran Yang, Yongxin Chen, and Zhaoran Wang.
\newblock Actor-critic provably finds nash equilibria of linear-quadratic
  mean-field games.
\newblock \emph{arXiv preprint arXiv:1910.07498}, 2019.

\bibitem[Gu et~al.(2021)Gu, Guo, Wei, and Xu]{gu2021mean}
Haotian Gu, Xin Guo, Xiaoli Wei, and Renyuan Xu.
\newblock Mean-field controls with q-learning for cooperative marl: convergence
  and complexity analysis.
\newblock \emph{SIAM Journal on Mathematics of Data Science}, 3\penalty0
  (4):\penalty0 1168--1196, 2021.

\bibitem[Gu et~al.(2016)Gu, Lillicrap, Sutskever, and Levine]{gu2016continuous}
Shixiang Gu, Timothy Lillicrap, Ilya Sutskever, and Sergey Levine.
\newblock Continuous deep q-learning with model-based acceleration.
\newblock In \emph{International conference on machine learning}, pages
  2829--2838. PMLR, 2016.

\bibitem[Guo et~al.(2019)Guo, Hu, Xu, and Zhang]{guo2019learning}
Xin Guo, Anran Hu, Renyuan Xu, and Junzi Zhang.
\newblock Learning mean-field games.
\newblock \emph{Advances in neural information processing systems}, 32, 2019.

\bibitem[Guo et~al.(2022)Guo, Xu, and Zariphopoulou]{guo2022entropy}
Xin Guo, Renyuan Xu, and Thaleia Zariphopoulou.
\newblock Entropy regularization for mean field games with learning.
\newblock \emph{Mathematics of Operations research}, 47\penalty0 (4):\penalty0
  3239--3260, 2022.

\bibitem[Hernandez-Leal et~al.(2019)Hernandez-Leal, Kartal, and
  Taylor]{hernandez2019survey}
Pablo Hernandez-Leal, Bilal Kartal, and Matthew~E Taylor.
\newblock A survey and critique of multiagent deep reinforcement learning.
\newblock \emph{Autonomous Agents and Multi-Agent Systems}, 33\penalty0
  (6):\penalty0 750--797, 2019.

\bibitem[Hu and Lauri{\`e}re(2023)]{hu2023recent}
Ruimeng Hu and Mathieu Lauri{\`e}re.
\newblock Recent developments in machine learning methods for stochastic
  control and games.
\newblock \emph{arXiv preprint arXiv:2303.10257}, 2023.

\bibitem[Jia and Zhou(2023)]{jia2023q}
Yanwei Jia and Xun~Yu Zhou.
\newblock q-learning in continuous time.
\newblock \emph{Journal of Machine Learning Research}, 24\penalty0
  (161):\penalty0 1--61, 2023.

\bibitem[Jiang and Jiang(2015)]{jiang2015global}
Yu~Jiang and Zhong-Ping Jiang.
\newblock Global adaptive dynamic programming for continuous-time nonlinear
  systems.
\newblock \emph{IEEE Transactions on Automatic Control}, 60\penalty0
  (11):\penalty0 2917--2929, 2015.

\bibitem[Kim and Yang(2020)]{kim2020hamilton}
Jeongho Kim and Insoon Yang.
\newblock Hamilton-jacobi-bellman equations for q-learning in continuous time.
\newblock In \emph{Learning for Dynamics and Control}, pages 739--748. PMLR,
  2020.

\bibitem[Kim et~al.(2021)Kim, Shin, and Yang]{Kim2021hjb}
Jeongho Kim, Jaeuk Shin, and Insoon Yang.
\newblock {Hamilton-Jacobi Deep Q-Learning for Deterministic Continuous-Time
  Systems with Lipschitz Continuous Controls}.
\newblock \emph{Journal of Machine Learning Research}, 22\penalty0
  (206):\penalty0 1--34, 2021.
\newblock URL \url{http://jmlr.org/papers/v22/20-1235.html}.

\bibitem[Kiran et~al.(2021)Kiran, Sobh, Talpaert, Mannion, Al~Sallab, Yogamani,
  and P{\'e}rez]{kiran2021deep}
B~Ravi Kiran, Ibrahim Sobh, Victor Talpaert, Patrick Mannion, Ahmad~A
  Al~Sallab, Senthil Yogamani, and Patrick P{\'e}rez.
\newblock Deep reinforcement learning for autonomous driving: A survey.
\newblock \emph{IEEE Transactions on Intelligent Transportation Systems},
  23\penalty0 (6):\penalty0 4909--4926, 2021.

\bibitem[Kloeden et~al.(2012)Kloeden, Platen, and Schurz]{kloeden2012numerical}
Peter~Eris Kloeden, Eckhard Platen, and Henri Schurz.
\newblock \emph{Numerical solution of SDE through computer experiments}.
\newblock Springer Science \& Business Media, 2012.

\bibitem[Kober et~al.(2013)Kober, Bagnell, and Peters]{kober2013reinforcement}
Jens Kober, J~Andrew Bagnell, and Jan Peters.
\newblock Reinforcement learning in robotics: A survey.
\newblock \emph{The International Journal of Robotics Research}, 32\penalty0
  (11):\penalty0 1238--1274, 2013.

\bibitem[Konda and Tsitsiklis(1999)]{konda1999actor}
Vijay Konda and John Tsitsiklis.
\newblock Actor-critic algorithms.
\newblock \emph{Advances in neural information processing systems}, 12, 1999.

\bibitem[Lasry and Lions(2007)]{Larsy2007}
Jean-Michel Lasry and Pierre-Louis Lions.
\newblock {Mean field games}.
\newblock \emph{Jpn. J. Math.}, 2\penalty0 (1):\penalty0 229--260, 2007.
\newblock \doi{https://doi.org/10.1007/s11537-007-0657-8}.

\bibitem[Lauri{\`e}re et~al.(2022)Lauri{\`e}re, Perrin, Geist, and
  Pietquin]{lauriere2022learning}
Mathieu Lauri{\`e}re, Sarah Perrin, Matthieu Geist, and Olivier Pietquin.
\newblock Learning mean field games: A survey.
\newblock \emph{arXiv preprint arXiv:2205.12944}, 2022.

\bibitem[Meyn and Tweedie(2012)]{meyn2012markov}
Sean~P Meyn and Richard~L Tweedie.
\newblock \emph{Markov chains and stochastic stability}.
\newblock Springer Science \& Business Media, 2012.

\bibitem[Mguni et~al.(2018)Mguni, Jennings, and
  de~Cote]{mguni2018decentralised}
David Mguni, Joel Jennings, and Enrique~Munoz de~Cote.
\newblock {Decentralised learning in systems with many, many strategic agents}.
\newblock In \emph{Thirty-Second AAAI Conference on Artificial Intelligence},
  2018.

\bibitem[Mnih et~al.(2013)Mnih, Kavukcuoglu, Silver, Graves, Antonoglou,
  Wierstra, and Riedmiller]{mnih2013playing}
Volodymyr Mnih, Koray Kavukcuoglu, David Silver, Alex Graves, Ioannis
  Antonoglou, Daan Wierstra, and Martin Riedmiller.
\newblock Playing atari with deep reinforcement learning.
\newblock \emph{arXiv preprint arXiv:1312.5602}, 2013.

\bibitem[Palanisamy et~al.(2014)Palanisamy, Modares, Lewis, and
  Aurangzeb]{palanisamy2014continuous}
Muthukumar Palanisamy, Hamidreza Modares, Frank~L Lewis, and Muhammad
  Aurangzeb.
\newblock Continuous-time q-learning for infinite-horizon discounted cost
  linear quadratic regulator problems.
\newblock \emph{IEEE transactions on cybernetics}, 45\penalty0 (2):\penalty0
  165--176, 2014.

\bibitem[Silver et~al.(2016)Silver, Huang, Maddison, Guez, Sifre, Van
  Den~Driessche, Schrittwieser, Antonoglou, Panneershelvam, Lanctot,
  et~al.]{silver2016mastering}
David Silver, Aja Huang, Chris~J Maddison, Arthur Guez, Laurent Sifre, George
  Van Den~Driessche, Julian Schrittwieser, Ioannis Antonoglou, Veda
  Panneershelvam, Marc Lanctot, et~al.
\newblock Mastering the game of go with deep neural networks and tree search.
\newblock \emph{nature}, 529\penalty0 (7587):\penalty0 484--489, 2016.

\bibitem[Subramanian and Mahajan(2019)]{SubramanianMahajan-2018-RLstatioMFG}
Jayakumar Subramanian and Aditya Mahajan.
\newblock {Reinforcement Learning in Stationary Mean-field Games}.
\newblock In \emph{Proceedings. 18th International Conference on Autonomous
  Agents and Multiagent Systems}, 2019.

\bibitem[Sutton and Barto(2018)]{sutton2018reinforcement}
Richard~S Sutton and Andrew~G Barto.
\newblock \emph{{Reinforcement learning: An introduction}}.
\newblock MIT press, 2018.

\bibitem[Tallec et~al.(2019)Tallec, Blier, and Ollivier]{tallec19a}
Corentin Tallec, L{\'e}onard Blier, and Yann Ollivier.
\newblock {Making Deep Q-learning methods robust to time discretization}.
\newblock In Kamalika Chaudhuri and Ruslan Salakhutdinov, editors,
  \emph{Proceedings of the 36th International Conference on Machine Learning},
  volume~97 of \emph{Proceedings of Machine Learning Research}, pages
  6096--6104. PMLR, 09--15 Jun 2019.

\bibitem[Vamvoudakis(2017)]{vamvoudakis2017q}
Kyriakos~G Vamvoudakis.
\newblock Q-learning for continuous-time linear systems: A model-free infinite
  horizon optimal control approach.
\newblock \emph{Systems \& Control Letters}, 100:\penalty0 14--20, 2017.

\bibitem[Wang et~al.(2020)Wang, Zariphopoulou, and Zhou]{wang2020reinforcement}
Haoran Wang, Thaleia Zariphopoulou, and Xun~Yu Zhou.
\newblock Reinforcement learning in continuous time and space: A stochastic
  control approach.
\newblock \emph{The Journal of Machine Learning Research}, 21\penalty0
  (1):\penalty0 8145--8178, 2020.

\bibitem[Wang et~al.(2021)Wang, Han, Yang, and Wang]{wang2021global}
Weichen Wang, Jiequn Han, Zhuoran Yang, and Zhaoran Wang.
\newblock Global convergence of policy gradient for linear-quadratic mean-field
  control/game in continuous time.
\newblock In \emph{International Conference on Machine Learning}, pages
  10772--10782. PMLR, 2021.

\bibitem[Watkins(1989)]{watkins1989learning}
Christopher~JCH Watkins.
\newblock Learning from delayed rewards.
\newblock \emph{PhD thesis, Cambridge University, Cambridge, England}, 1989.

\bibitem[Watkins and Dayan(1992)]{watkins1992q}
Christopher~JCH Watkins and Peter Dayan.
\newblock Q-learning.
\newblock \emph{Machine learning}, 8:\penalty0 279--292, 1992.

\bibitem[Williams(1992)]{williams1992simple}
Ronald~J Williams.
\newblock Simple statistical gradient-following algorithms for connectionist
  reinforcement learning.
\newblock \emph{Machine learning}, 8:\penalty0 229--256, 1992.

\bibitem[Yang et~al.(2018)Yang, Luo, Li, Zhou, Zhang, and Wang]{yang2018mean}
Yaodong Yang, Rui Luo, Minne Li, Ming Zhou, Weinan Zhang, and Jun Wang.
\newblock Mean field multi-agent reinforcement learning.
\newblock In \emph{International conference on machine learning}, pages
  5571--5580. PMLR, 2018.

\bibitem[Yang et~al.(2019)Yang, Chen, Hong, and Wang]{yang2019provably}
Zhuoran Yang, Yongxin Chen, Mingyi Hong, and Zhaoran Wang.
\newblock Provably global convergence of actor-critic: A case for linear
  quadratic regulator with ergodic cost.
\newblock \emph{Advances in neural information processing systems}, 32, 2019.

\bibitem[Zaman et~al.(2023)Zaman, Koppel, Bhatt, and Basar]{zaman2023oracle}
Muhammad Aneeq~Uz Zaman, Alec Koppel, Sujay Bhatt, and Tamer Basar.
\newblock {Oracle-free reinforcement learning in mean-field games along a
  single sample path}.
\newblock In \emph{International Conference on Artificial Intelligence and
  Statistics}, pages 10178--10206. PMLR, 2023.

\bibitem[Zeng et~al.(2021)Zeng, Doan, and Romberg]{zeng2021two}
Sihan Zeng, Thinh~T Doan, and Justin Romberg.
\newblock A two-time-scale stochastic optimization framework with applications
  in control and reinforcement learning.
\newblock \emph{arXiv preprint arXiv:2109.14756}, 2021.

\bibitem[Zhang et~al.(2021)Zhang, Yang, and Ba{\c{s}}ar]{zhang2021multi}
Kaiqing Zhang, Zhuoran Yang, and Tamer Ba{\c{s}}ar.
\newblock Multi-agent reinforcement learning: A selective overview of theories
  and algorithms.
\newblock \emph{Handbook of reinforcement learning and control}, pages
  321--384, 2021.

\bibitem[Zhou and Lu(2023)]{ZhouLu2023}
Mo~Zhou and Jianfeng Lu.
\newblock {Single Timescale Actor-Critic Method to Solve the Linear Quadratic
  Regulator with Convergence Guarantees}.
\newblock \emph{Journal of Machine Learning Research}, 24\penalty0
  (222):\penalty0 1--34, 2023.
\newblock URL \url{http://jmlr.org/papers/v24/22-0644.html}.

\end{thebibliography}

\end{document}